\documentclass[12pt, a4paper]{amsart}
\usepackage{amssymb}
\usepackage{tikz} 
\usepackage{mathabx}
\usepackage{graphicx}
\usepackage{mathrsfs}
\usepackage{here}
\DeclareFontFamily{OT1}{rsfs}{}
\DeclareFontShape{OT1}{rsfs}{n}{it}{<-> rsfs10}{}
\DeclareMathAlphabet{\mathscr}{OT1}{rsfs}{n}{it}

\theoremstyle{plain}

\newtheorem{theorem}{Theorem}[section]
\newtheorem{lemma}[theorem]{Lemma}
\newtheorem{proposition}[theorem]{Proposition}

\theoremstyle{definition}

\newtheorem{example}{Example}[section]

\theoremstyle{remark}
\newtheorem{remark}[theorem]{Remark}%numbering within section

\def\Re{\mathrm{Re}}
\def\Im{\mathrm{Im}}

\newtheorem*{acknowledgement}{Acknowledgement}

\begin{document}

\date\today

\title[Nonminimal infinite type models in $\mathbb C^2$]
{Infinitesimal CR automorphisms and stability groups of nonminimal infinite type 
models in $\mathbb C^2$}
\author[V.~T.~Ninh]{Van Thu Ninh\textit{$^{1,2}$}}
\author[T.~N.~O.~Duong]{Thi Ngoc Oanh Duong}
\author[V.~H.~Pham]{Van Hoang Pham}
\author[H.~Kim]{Hyeseon Kim}

\address{(V.~T.~Ninh)}
\address{\textit{$^{1}$}~Department of Mathematics,   VNU University of Science, Vietnam National University at Hanoi, 334 Nguyen Trai, Thanh Xuan, Hanoi, Vietnam}
 \address{\textit{$^{2}$}~Thang Long Institute of Mathematics and Applied Sciences,
Nghiem Xuan Yem, Hoang Mai, HaNoi, Vietnam}
\email{thunv@vnu.edu.vn}

\address[T.~N.~O.~Duong]{Phung Hung Secondary School, 55 Nguyen Thai Hoc, Quang Trung, Son Tay, Hanoi, Vietnam}
\email{duongoanh25.6@gmail.com}

\address[V.~H.~Pham]{Department of Mathematics, Vietnam
National University at Hanoi, 334 Nguyen Trai str., Hanoi, Vietnam}
\email{hoangabcmnpxyz95@gmail.com}

\address[H.~Kim]{Research Institute of Mathematics, Seoul National University, 1 Gwanak-ro, Gwanak-gu, Seoul 08826, Republic of Korea}
\email{hop222@snu.ac.kr}


\begin{abstract}
We determine infinitesimal $\mathrm{CR}$ automorphisms and stability groups of real hypersurfaces in $\mathbb C^2$ in the case when the hypersurface is nonminimal and of infinite type at the reference point.
\end{abstract}

\maketitle

\section{Introduction and the statement of main results}
The purpose of this article is to describe the spaces of infinitesimal $\mathrm{CR}$ automorphisms and stability groups of real hypersurfaces in $\mathbb C^2$ such that they are nonminimal in the sense of Tumanov \cite{Tum89} and of infinite type at the origin in the sense of D'Angelo \cite{Dan82}. 

We now introduce some notations which are needed to state our main results. Let $(M,p)$ be the germ at $p$ of a $C^{\infty}$-smooth real hypersurface $M$ in $\mathbb C^n$, $n\geq2$. We denote by $\mathrm{Aut}(M)$ the $\mathrm{CR}$ automorphism group of $M$. For each $p\in M$, we denote by $\mathrm{Aut}(M,p)$ the set of germs at $p$ of biholomorphisms mapping $M$ into itself and fixing the point $p$. In addition, we denote by $\mathfrak{aut}(M,p)$ the set of germs of holomorphic vector fields in $\mathbb C^n$ at $p$ whose real part is tangent to $M$. With this notation, a smooth vector field germ $(X,p)$ on $M$ is called an \emph{infinitesimal $\mathrm{CR}$ automorphism germ at $p$ of $M$} if there exists an element in $\mathfrak{aut}(M,p)$ such that its real part is equal to $X$ on $M$. We also denote by $\mathfrak{aut}_{0}(M,p)$ the set of all elements $H\in\mathfrak{aut}(M,p)$ for which $H$ vanishes at $p$.  

The study of $\mathrm{CR}$ geometry on real hypersurfaces in $\mathbb C^n$ is relatively well-developed in the case of \emph{rigid hypersurfaces} (see \cite{KL2015}, \cite{KosL2016}, \cite{Sta95} and the references therein). Here, we say that a $C^{\infty}$-smooth real hypersurface $M$ through the origin in $\mathbb C^n$ is \emph{rigid} if there exist coordinates $(z,w)\in\mathbb C^{n-1}\times\mathbb C$ and a $C^{\infty}$-smooth function $F$ near the origin such that $M$ is given by an equation of the form
\begin{equation}\label{rigid equation}
\mathrm{Re}\;w=F(z,\bar z)
\end{equation}(cf. \cite{BRT85} and \cite{Sta95}). For a certain class of rigid hypersurfaces of finite type in the sense of D'Angelo in $\mathbb C^2$, we refer the reader to \cite{Sta91} which addresses the existence of infinitesimal $\mathrm{CR}$ automorphisms. However, if we move our attention to the case of \emph{rigid hypersurfaces of infinite type}, then we necessarily encounter more complicated procedure to get such geometric object due to the computational difficulty and the lack of literatures in the setting of infinite type (see \cite{HN2016} and the references therein). As a significant result which has inspired the present paper, Hayashimoto and Ninh \cite{HN2016} investigated an \emph{infinite type model} $(M'_P,0)$ in $\mathbb C^2$ which is defined by
\begin{equation}\label{HN model hypersurface}
M'_{P}:=\left\{(z,w)\in\mathbb C^2: \mathrm{Re}\;w+ P(z)=0\right\},
\end{equation}where $P$ is a non-zero germ of a real-valued $C^{\infty}$-smooth function at the origin vanishing to infinite order at $z=0$. More precisely, the associated $\mathrm{Aut}(M'_P,0)$, $\mathfrak{aut}(M'_P,0)$, $\mathfrak{aut}_{0}(M'_P,0)$ were explicitly described under the variance of the zero set of the function $P$ defined in \eqref{HN model hypersurface}. Furthermore, it follows from the definition that $M'_P$ given in \eqref{HN model hypersurface} is a rigid real hypersurface of infinite type. 

We now employ the concept of a nonminimal hypersurface (this term is coined in \cite{Tum89}) which has also inspired the present paper. By following the definition in \cite{Tum89}, a $\mathrm{CR}$ manifold $N$ is \emph{minimal at a point $p\in N$} if there are no submanifolds passing through $p$ of smaller dimension but with the same $\mathrm{CR}$ dimension. In this sense, one can say that a real hypersurface $N\in\mathbb C^2$ is \emph{nonminimal at a point $p\in N$} if there exists a germ of a complex hypersurface $E$ through $p$ which is contained in $N$ (cf. \cite{JL2013} and \cite{KKZ2017}). In addition, a germ at the origin of a real hypersurface $(N,0)$ in $\mathbb C^2$ is a \emph{ruled hypersurface} if there exist coordinates $(z,w)\in\mathbb C^2$ such that $N$ is given by an equation of the form
\[
\mathrm{Im}\;w=(\mathrm{Re}\;w) A(z,\bar z),
\]where $A(z,\bar z)$ does not vanish identically (for more details on $A$ in the case when $N$ is a ruled real analytic hypersurface of infinite type, see Eq.~$(6)$ and the consecutive arguments in \cite[Section~$3$]{KL2015}). Moreover, a ruled hypersurface is known as a crucial prototype in considering local equivalence problem of nonminimal real analytic hypersurfaces in $\mathbb C^2$. We further say that a germ at $p$ of a real hypersurface $(N,p)$ in $\mathbb C^2$ is \emph{$m$-nonminimal} ($m\geq1$) at $p$ if there exist local coordinates $(z,w)\in\mathbb C^2$, $p$ corresponds to $0$, close by $0$, such that $N$ is given by an equation of the form
\begin{equation}\label{definition of m-nonminimal condition}
\mathrm{Im}\;w={(\mathrm{Re}\;w)}^m\psi(z,\bar z,\mathrm{Re}\;w),
\end{equation}where $\psi(z,0,\mathrm{Re}\;w)=\psi(0,\bar z,\mathrm{Re}\;w)=0$ and $\psi(z,\bar z,0)$ does not vanish identically (cf. \cite{JL2013} and \cite{KosL2016}). In particular, if $(N,p)$ is a germ at $p$ of a real analytic hypersurface which is $1$-nonminimal at $p$, then $\mathrm{Aut}(N,p)$ constitutes a finite dimensional Lie group (see \cite[Theorem~$1$]{JL2013}). Moreover, a class of real analytic $1$-nonminimal hypersurface in $\mathbb C^2$ is also meaningful in the sense that such nonminimal condition is related to the degeneration of the Levi form as the natural second-order invariant of a real analytic hypersurface (cf. \cite[Introduction]{JL2013}).

In this paper, we first investigate the spaces of infinitesimal $\mathrm{CR}$ automorphisms and stability groups of a $1$-nonminimal infinite type model $(M_P,0)$ in $\mathbb C^2$ which is defined by 
\begin{equation}\label{our 1-nonminimal infinite type model}
M_{P}:=\left\{(z_1,z_2)\in\mathbb C^2: \mathrm{Re}\;z_1+(\mathrm{Im}\;z_1)P(z_2)=0\right\},
\end{equation}where $P$ is a non-zero germ of a real-valued $C^{\infty}$-smooth function at $0$ vanishing to infinite order at $z_2=0$.

Before stating our main results, we now prepare further notations. For each $r>0$, let us denote by $\triangle_{r}$ the complex disk of radius $r$ centred at the origin in $\mathbb C$. We also denote by $\triangle^{*}_{r}$ the punctured disk $\triangle_{r}\setminus\{0\}$. For a sufficiently small $\epsilon_{0}>0$ and a $C^{\infty}$-smooth function $P:\triangle_{\epsilon_0}\rightarrow\mathbb R$, we denote by $S_{\infty}(P)$ the set of all points $z\in\triangle_{\epsilon_0}$ for which $\nu_{z}(P)=+\infty$, where $\nu_{z}(P)$ is the vanishing order of $P(z+\zeta)-P(z)$ at $\zeta=0$. In addition, we denote by $P_{\infty}(M_P)$ the set of all points of infinite type in $M_{P}$. We note that it is not hard to see that 
\begin{equation}\label{P_infty(our 1-nonminimal infinite type model)}
P_{\infty}(M_P)\supset\left\{(it-tP(z_2),z_2): t\in\mathbb R, z_2\in S_{\infty}(P)\right\}.
\end{equation}
However, we could not have the equality (see Example \ref{0-e.g.} in Section \ref{two examples}).  A similar example to Example \ref{0-e.g.} shows that in general the equality  
$$
P_{\infty}(M'_P)=\left\{(it-P(z_2),z_2): t\in\mathbb R, z_2\in S_{\infty}(P)\right\}
$$
also could not hold for the rigid infinite type model $(M'_{P},0)$ (see \cite[Remark~$1$]{HN2016}). Therefore, we think that the assumption ``the connected component of $0$ in $S_\infty(P)$ is $\{0\}$" given in the statements of \cite[Theorem $1$ and Theorem $2$]{HN2016} should be replaced by ``the connected component of $(0,0)$ in $P_\infty(M'_P)$ is $\{(it, 0)\in \mathbb C^2\colon t\in \mathbb R\}$".

We now ready to state our main results. For the case of a $1$-nonminimal infinite type model, we have three main theorems in this paper. Theorem~\ref{main THM1} comes under the case that special conditions on holomorphic vector fields determine the precise form of local defining functions. The other two main theorems explain the converse situation. Such division on the main results is originated in the work of Hayashimoto and Ninh \cite{HN2016}. In what follows, as commented in \cite[Introduction]{HN2016}, all functions, mappings, hypersurfaces, and so on, will be understood as germs at the reference points unless stated otherwise.

\begin{theorem}\label{main THM1}
Let $(M_P,0)$ be a $C^\infty$-smooth hypersurface in $\mathbb C^2$ defined by the equation $\rho(z)=\rho(z_1,z_2):=\Re\;z_1+(\Im\;z_1)P(z_2)=0$, where $P$ is a $C^\infty$-smooth function on a neighborhood of the origin in $\mathbb C$ satisfying:
\begin{enumerate}
\item[$(i)$] The connected component of $z_2=0$ in the zero set of $P$ is $\{0\}$;
\item[$(ii)$] $P$ vanishes to infinite order at $z_2=0$.
\end{enumerate}
Then any holomorphic vector field vanishing at the origin tangent to $(M_P,0)$ is either of the form $\alpha z_1\partial z_1$ for some $\alpha\in\mathbb R$, or after a change of variable in $z_2$, of the form $\alpha z_1\partial z_1+i\beta z_2\partial z_2$ for some $\alpha\in\mathbb R$ and  $\beta\in\mathbb R^{*}$, in which case $M_P$ is rotationally symmetric, that is, $P(z_2)\equiv P(|z_2|)$.
\end{theorem}
 
\begin{remark}
The condition $(i)$ in Theorem~\ref{main THM1} simply shows that the set $\{z_2\in\mathbb C: P(z_2)=0\}$ does not contain any curve in $\mathbb C$. In contrast to this theorem, Theorem~\ref{main THM3} below allows the possibility that the curve $\mathrm{Re}\;z_2=0$ is contained in the zero set of $P$. Moreover, the condition $(ii)$ and consideration of the points given in \eqref{P_infty(our 1-nonminimal infinite type model)} provide a first step for the proof of Theorem~\ref{main THM1}.
\end{remark}

\begin{theorem}\label{main THM2}
Let $(M_P,0)$ be a $C^\infty$-smooth hypersurface in $\mathbb C^2$ defined by the equation $\rho(z)=\rho(z_1,z_2):=\Re\;z_1+(\Im\;z_1)P(z_2)=0$, where $P$ is a $C^\infty$-smooth function on a neighborhood of $0$, vanishing to infinite order at $z_2=0$, and satisfying:
\begin{enumerate}
\item[$(i)$] $P(z_2)\notequiv0$ on a neighborhood of $z_2=0$;
\item[$(ii)$] the connected component of $(0,0)$ in $P_\infty(M_P)$ is $\{(it, 0)\in \mathbb C^2\colon t\in \mathbb R\}$.
\end{enumerate}Then the following assertions hold:
\begin{enumerate}
\item[$(a)$] $\mathfrak{aut}(M_{P},0)=\mathfrak{aut}_{0}(M_{P},0)$.
\item[$(b)$] If $\mathfrak{aut}_{0}(M_{P},0)=\{\alpha z_1\partial z_1:\alpha\in\mathbb R\}$, then
\[
\mathrm{Aut}(M_{P},0)=G_{2}(M_{P},0),
\]where $G_{2}(M_{P},0)$ is the set of all $\mathrm{CR}$ automorphisms of $M_{P}$ defined by
\[
(z_1,z_2)\mapsto (Cz_1,g_{2}(z_2))
\]for some constant $C\in\mathbb R^{*}$ and some holomorphic function $g_2$ with $g_{2}(0)=0$ and $|g'_{2}(0)|=1$ defined on a neighborhood of the origin in $\mathbb C$ satisfying that $P(g_2(z_2))\equiv P(z_2)$.
\end{enumerate}
\end{theorem}

\begin{remark}\label{connected-component}
Suppose that the connected component of $(0,0)$ in $P_{\infty}(M_P)$ is $\{(it,0): t\in\mathbb R\}$. Then, by definition, the connected component of $0$ in $S_{\infty}(P)$ is just $\{0\}$. This fact provides a crucial ingredient in the proof of Theorem~\ref{main THM2}.
\end{remark}

In the case when $S_{\infty}(P)$ contains a non-trivial connected component of $(0,0)$ which contrasts with the condition $(ii)$ of Theorem~\ref{main THM2}, for instance $M_P$ is \emph{tubular}, we obtain the following theorem.

\begin{theorem}\label{main THM3}
Let $\widetilde P$ be a $C^{\infty}$-smooth function defined on a neighborhood of $0$ in $\mathbb C$ satisfying:
\begin{enumerate}
\item[$(i)$] $\widetilde P(x)\notequiv0$ on a neighborhood of $x=0$ in $\mathbb R$;
\item[$(ii)$] the connected component of $(0,0)$ in $P_\infty(M_{\widetilde P})$ is $\{(it, 0)\in \mathbb C^2\colon t\in \mathbb R\}$.
\end{enumerate}Denote by $P$ a function defined by setting $P(z_2):=\widetilde{P}(\Re\;z_2)$ with a further condition that $P(z_2)$ vanishes to infinite order at $z_2=0$. Then the following assertions hold:
\begin{enumerate}
\item[$(a)$] $\mathfrak{aut}_{0}(M_{P},0)=\{\alpha z_1\partial_{z_1}: \alpha\in\mathbb R\}$ and the Lie algebra $\mathfrak{g}=\mathfrak{aut}(M_{P},0)$ admits the decomposition
\[
\mathfrak{g}=\mathfrak{g}_1\oplus\mathfrak{g}_0,
\]where $\mathfrak{g}_1=\{\alpha z_1\partial_{z_1}:\alpha\in\mathbb R\}$  and $\mathfrak{g}_0=\{i\beta\partial_{z_2}:\beta\in\mathbb R\}$.
\item[$(b)$] $\mathrm{Aut}(M_{P},0)$ is either $\{(z_1,z_2)\mapsto (tz_1,z_2): t\in\mathbb R^{*}\}$ or $\{(z_1,z_2)\mapsto (tz_1,\pm z_2): t\in\mathbb R^{*}\}$, where the latter case happens only if $P(z_2)=P(-z_2)$.
\item[$(c)$] If $P_\infty(M_P)=\{(it, is)\in \mathbb C^2\colon t,s\in \mathbb R\}$, then $\mathrm{Aut}(M_{P})$ can be decomposed into either 
\[
T^{1}(M_{P})\oplus T^{2}(M_{P})
\]or
\[
T^{2}(M_{P})\oplus T^{3}(M_{P}),
\]where $T^1(M_{P})=\{(z_1,z_2)\mapsto (s z_1,z_2): s\in\mathbb R^{*}\}$, $T^{2}(M_{P})=\{(z_1,z_2)\mapsto (z_1,z_2+it): t\in\mathbb R\}$ and $T^3(M_{P})=\{(z_1,z_2)\mapsto (s z_1,\pm z_2): s\in\mathbb R^{*}\}$. The latter case happens only if $P(z_2)=P(-z_2)$.
\end{enumerate}
\end{theorem}

In addition, we also investigate an analogue of Theorem~\ref{main THM1} for an $m$-nonminimal infinite type model $(M_{P,m},0)$ with $m>1$ in $\mathbb C^2$ which is defined by
\begin{equation}\label{our m-nonminimal infinite type model}
M_{P,m}:=\left\{(z_1,z_2)\in\mathbb C^2: \mathrm{Im}\;z_1-{(\mathrm{Re}\;z_1)}^mP(z_2)=0\right\},
\end{equation}where $P$ is a non-zero germ of a real-valued $C^{\infty}$-smooth function at the origin, which vanishes to infinite order at $z_2=0$. Due to the variance of the choice of the constant $m$ in \eqref{our m-nonminimal infinite type model}, the procedure to analyze the associated holomorphic vector fields becomes more complicated than that of a $1$-nonminimal infinite type model $(M_P,0)$ defined above. For the convenience of exposition, we shall proceed the assertion for the case of $(M_{P,m},0)$  with $m>1$ separately in Appendix.

The organization of the paper is described as follows: In Section~\ref{section for the proof of main THM1}, we provide the proof of Theorem~\ref{main THM1} for which certain conditions on holomorphic vector fields determine the precise form of local defining functions. As the converse of this situation, we next provide the proofs of Theorem~\ref{main THM2} and Theorem~\ref{main THM3} in Section~\ref{section for the proofs of main THM2 and main THM3}. In Section~\ref{two examples}, we first elaborate a counterexample in addressing the significance about the converse inclusion of Eq.~\eqref{P_infty(our 1-nonminimal infinite type model)}. We further present several examples in the same section as analogues of those in \cite{HN2016}. In addition, an analogue of Theorem~\ref{main THM1} for an $m$-nonminimal infinite type model $(M_{P,m},0)$ with $m>1$ will be investigated in Appendix.

\section{Analysis of holomorphic tangent vector fields}\label{section for the proof of main THM1}
This section is devoted to the proof of our first main result Theorem~\ref{main THM1}. Let us first prepare two crucial technical ingredients for the proof of Theorem~\ref{main THM1}. The following proposition will be treated also in the assertion for an $m$-nonminimal infinite type model $(M_{P,m},0)$ with $m>1$ in Appendix.

\begin{proposition}[{\cite[Lemma~$7$]{HN2016}}]\label{technical lemma1:THM1}
Let $P:\triangle_{\epsilon_0}\rightarrow\mathbb R$ be a $C^\infty$-smooth function satisfying that the connected component of $z=0$ in the zero set of $P$ is $\{0\}$ and that $P$ vanishes to infinite order at $z=0$. If $a,b$ are complex numbers and if $g_0, g_1, g_2$ are $C^\infty$-smooth functions defined on $\triangle_{\epsilon_0}$ satisfying:
\begin{enumerate}
\item[$(\mathrm A1)$] $g_0(z)=O(|z|), g_1(z)=O({|z|}^\ell)$, and $g_2(z)=o({|z|}^m)$;
\item[$(\mathrm A2)$] $\Re\left[(az^m+g_2(z))P^{n+1}(z)+bz^{\ell}(1+g_0(z))P_z(z)+g_1(z)P(z)\right]\equiv0$ on $\triangle_{\epsilon_0}$
\end{enumerate}for any non-negative integers $l, m$ and $n$ except for the following two cases
\begin{enumerate}
\item[$(\mathrm E1)$] $\ell=1$ and $\Re\;b=0$;
\item[$(\mathrm E2)$] $m=0$ and $\Re\;a=0$,
\end{enumerate}then $ab=0$.
\end{proposition}
The proof of this proposition proceeds along the similar lines as that of Lemma~$3$ in~\cite{KN2015}.~(Notice that $P$ was assumed to be positive on $\triangle^{*}_{\epsilon_0}$ in~\cite{KN2015}.) For the sake of brevity we shall omit routine arguments, except \eqref{crucial o.d.e. in technical lemma2} below. The following lemma assures the existence of a modification of Eq.~$(7)$ in~\cite{KN2015}, which is a main ingredient for the proof.
For the convenience of the reader, we provide the proof of the following lemma.
\begin{lemma}[{\cite[Lemma~$8$]{HN2016}}]\label{technical lemma2:LEM1}
Let $P, a, b, g_0, g_1, g_2$ be as in Proposition~\ref{technical lemma1:THM1}. Suppose that for each $t_0\in\mathbb R$, $\gamma: [t_{0},t_{\infty})\rightarrow \triangle^{*}_{\epsilon_{0}}$, where $t_{\infty}$ satisfies either $t_\infty\in\mathbb R$ or $t_\infty=+\infty$, is a solution of the initial-value problem
\begin{equation}\label{crucial o.d.e. in technical lemma2}
\dfrac{d\gamma}{dt}(t)=b\gamma^{\ell}(t)(1+g_{0}(\gamma(t))),\quad\gamma(t_0)=z_0,
\end{equation}where $z_0\in\triangle^{*}_{\epsilon_0}$ with $P(z_0)\neq0$, such that $\lim_{t\uparrow t_{\infty}}\gamma(t)=0$. Then $P(\gamma(t))\neq0$ for all $t\in(t_0,t_\infty)$.
\end{lemma}
\begin{proof}
Aiming for a contradiction, we suppose that $P$ has a zero on the curve $\gamma$. Then since the connected component of $z=0$ in the zero set of $P$ is $\{0\}$, without loss of generality, we may further assume that there exists a $t_1\in(t_0,t_\infty)$ such that $P(\gamma(t))\neq0$ for all $t\in(t_0,t_1)$ and $P(\gamma(t_1))=0$. 

Let $u(t):=\frac{1}{2}\log|P(\gamma(t))|$ for $t_0<t<t_1$. Then it follows from \eqref{crucial o.d.e. in technical lemma2} and $(\mathrm A2)$ that
\[
u'(t)=-P^{n}(t)\left(\Re(a\gamma^{m}(t)+o({|\gamma(t)|}^m))\right)+O({|\gamma(t)|}^\ell)
\]for all $t_0<t<t_1$. Combining this with the assumption for the vanishing order of $P$ at $z=0$, one can deduce that $u'(t)$ is bounded on $(t_0,t_1)$. This after applying the fundamental theorem of ordinary differential equations in turn yields the boundedness of $u(t)$ on $(t_0,t_1)$, which is absurd since $u(t)\rightarrow-\infty$ as $t\uparrow t_1$. Hence our proof is complete.
\end{proof}

\medskip

Before going further, we shall fix the notations. In what follows, we denote by $\mathbb N^{0}$ and $\mathbb N^{*}$ the set of all non-negative integers and the set of all positive integers, respectively.

\subsection{Proof of Theorem~\ref{main THM1}}

The $\mathrm{CR}$ hypersurface germ $(M_P,0)$ at the origin in $\mathbb C^2$ under consideration is defined by the equation
\[
\rho(z)=\rho(z_1,z_2):=\Re~z_1+(\Im~z_1)P(z_2)=0,
\]where $P$ is a $C^\infty$-smooth function satisfying the two above conditions $(i)$ and $(ii)$. Then we consider a holomorphic vector field $H=h_{1}(z_1,z_2)\partial z_1+h_{2}(z_1,z_2)\partial z_2$ defined near the origin in $\mathbb C^2$. We focus only on $H$ which is tangent to $M_P$. This means that $H$ satisfies the identity 
\begin{equation}\label{tangency condition in main THM1}
(\Re~H)\rho(z)=0,\quad\forall z\in M_{P}.
\end{equation}Expanding $h_1$ and $h_2$ into the Taylor series at the origin, we get
\begin{align*}
h_{1}(z_1,z_2)&=\sum_{j,k=0}^{\infty}a_{j,k}z^{j}_{1}z^{k}_{2}=\sum_{j=0}^{\infty}a_{j}(z_2)z^{j}_{1};\\
h_{2}(z_1,z_2)&=\sum_{j,k=0}^{\infty}b_{j,k}z^{j}_{1}z^{k}_{2}=\sum_{j=0}^{\infty}b_{j}(z_2)z^{j}_{1},
\end{align*}where $a_{j,k}, b_{j,k}\in\mathbb C$ and $a_j, b_j$ are holomorphic functions for all $j\in\mathbb N^{0}$. Moreover, we further assume that $H(0,0)=0$. Then it follows that
\[
a_{0,0}=b_{0,0}=0
\]since $h_{1}(0,0)=h_{2}(0,0)=0$. A direct computation shows that 
\[
\rho_{z_1}(z_1,z_2)=\frac{1}{2}+\frac{1}{2i}P(z_2);\quad\rho_{z_2}(z_1,z_2)=(\Im~z_1)P_{z_2}(z_2),
\]and hence \eqref{tangency condition in main THM1} can be re-written as
\begin{equation*}\label{re-written1 tangency condition in main THM1}
\Re\left[\left(\frac{1}{2}+\frac{1}{2i}P(z_2)\right)h_{1}(z_1,z_2)+(\Im~z_1)P_{z_2}(z_2)h_{2}(z_1,z_2)\right]=0
\end{equation*}for all $(z_1,z_2)\in M_{P}$.

Since $(it-tP(z_2),z_2)\in M_{P}$ with $t\in\mathbb R$ small enough, the previous equation again admits a new form
\begin{equation}\label{re-written2 tangency condition in main THM1}
\Re\left[\left(\frac{1}{2}+\frac{1}{2i}P(z_2)\right)\sum_{j,k=0}^{\infty}a_{j,k}{(it-tP(z_2))}^jz^{k}_{2}+tP_{z_2}(z_2)\sum_{m,n=0}^{\infty}b_{m,n}{(it-tP(z_2))}^mz^{n}_{2}\right]=0
\end{equation}for all $z_2\in\mathbb C$ and $t\in\mathbb R$ with $z_2\in\triangle_{\epsilon_0}$ and $|t|<\delta_0$, where $\epsilon_0, \delta_0>0$ are small enough.

Inserting $t=0$ into \eqref{re-written2 tangency condition in main THM1}, we have
\begin{equation}\label{the first induced condition in main THM1}
\Re\left[\left(\frac{1}{2}+\frac{1}{2i}P(z_2)\right)\sum_{k=0}^{\infty}a_{0,k}z^{k}_{2}\right]\equiv0
\end{equation}on $\triangle_{\epsilon_0}$. Combining this with the assumption that $P$ vanishes to infinite order at $z_2=0$, one can assert that 
\begin{equation}\label{the first condition on a in main THM1}
a_{0,k}=0,\quad\forall k\in\mathbb N^{*}.
\end{equation}Moreover, setting the coefficient of $t^{m+1}$ in \eqref{re-written2 tangency condition in main THM1} equals zero for each $m\in\mathbb N^{0}$, we obtain
\begin{equation}\label{the 2nd condition induced in main THM1}
\Re\left[\left(\frac{1+P^{2}(z_2)}{2}\right)\sum_{k=0}^{\infty}i a_{m+1,k}{(i-P(z_2))}^{m}z^{k}_{2}+P_{z_2}(z_2)\sum_{n=0}^{\infty}b_{m,n}{(i-P(z_2))}^{m}z^{n}_{2}\right]\equiv0
\end{equation}for each $m\in\mathbb N^{0}$ on $\triangle_{\epsilon_0}$. Since both $P(z_2)$ and $P_{z_2}(z_2)$ vanish to infinite order at $z_2=0$, \eqref{the 2nd condition induced in main THM1} yields
\begin{equation}\label{the 3rd condition induced in main THM1}
\Re\left[\sum_{k=0}^{\infty}i^{m+1}a_{m+1,k}z^{k}_{2}\right]\equiv0
\end{equation}for each $m\in\mathbb N^{0}$ on $\triangle_{\epsilon_0}$. Then it follows from \eqref{the first condition on a in main THM1} and \eqref{the 3rd condition induced in main THM1} that
\begin{equation}\label{conditions on a in main THM1}
a_{j,k+1}=0,\;\forall j,k\in\mathbb N^{0};\;\Re(i^{\ell} a_{\ell,0})=0,\;\forall\ell\in\mathbb N^{*}.
\end{equation}Considering again the assumption for the vanishing order of $P(z_2)$ at $z_2=0$, we indeed have 
\[
h_{1}(z_1,z_2)=\alpha z_1
\]for some $\alpha\in\mathbb R$, if $h_{2}(z_1,z_2)\equiv0$. Therefore, in the remaining of the proof, we always assume that $h_2\notequiv0$ without loss of generality.

Let $m_0$ be the smallest integer such that $b_{m_0,n}\neq0$ for some $n\in\mathbb N^{0}$. Then we let $n_0$ be the smallest integer such that $b_{m_0,n_0}\neq0$. Since $b_{0,0}=0$, it is clear that $m_0\geq1$ if $n_0=0$. With this setting, \eqref{the 2nd condition induced in main THM1}  and \eqref{conditions on a in main THM1} yield
\begin{equation}\label{the 4th condition induced in main THM1}
\Re\left[\left(\frac{1+P^{2}(z_2)}{2}\right)ia_{m_0+1,0}{(i-P(z_2))}^{m_0}+P_{z_2}(z_2)\sum_{n=n_0}^{\infty}b_{m_0,n}{(i-P(z_2))}^{m_0}z^{n}_{2}\right]\equiv0
\end{equation}on $\triangle_{\epsilon_0}$. Since $P(z_2)=o({|z_2|}^j)$ for any $j\in\mathbb N^{*}$, it follows from \eqref{the 4th condition induced in main THM1} that
\[
\Re\left[\left(\frac{1+P^{2}(z_2)}{2}\right)ia_{m_0+1,0}{(i-P(z_2))}^{m_0}+i^{m_0}b_{m_0,n_0}(z^{n_0}_{2}+o(z^{n_0}_{2}))P_{z_2}(z_2)\right]\equiv0
\]on $\triangle_{\epsilon_0}$.

Now we shall consider the following two cases.

\medskip

\noindent{\bf Case~1.} $m_0=0$. In this case, by \cite[Corollary~$4$]{HN2016}, we first obtain $n_0=1$ and $b_{0,1}=i\beta$ for some $\beta\in\mathbb R^{*}$. Then, by a change of variables~(cf.~\cite[Lemma~$1$]{N2013}), we may assume that
\[
b_{0}(z_2)=\sum_{n=0}^{\infty}b_{0,n}z^{n}_{2}=i\beta z_2.
\]Therefore, we get from \eqref{the 4th condition induced in main THM1} that
\begin{equation}\label{the 5th condition induced in main THM1}
\Re\left[i\beta z_2 P_{z_2}(z_2)\right]\equiv0
\end{equation}on $\triangle_{\epsilon_0}$. This implies that $P(z_2)\equiv P(|z_2|)$ on $\triangle_{\epsilon_0}$.

We now prove that $b_m=0$ for every $m\in\mathbb N^{*}$. Suppose otherwise. Then there exists the smallest number $m_1\in\mathbb N^{*}$ such that $b_{m_1}\notequiv0$. By the same argument as above, we may assume that $b_{m_1}(z_2)\equiv i^{1-m_1}\beta_1 z_2+o(|z_2|)$ for some $\beta_1\in\mathbb R^{*}$ on $\triangle_{\epsilon_0}$. Moreover, we indeed have $b_{m_1}(z_2)=i^{1-m_1}\beta_1 z_2$ for some $\beta_1\in\mathbb R^{*}$: suppose otherwise. Then there exist $k_0\geq2$ and $c_{k_0}\in\mathbb C^{*}$ such that
\[
b_{m_1}(z_2)=i^{1-m_1}\beta_1 z_2+c_{k_0} z^{k_0}_{2}+o({|z_2|}^{k_0}).
\]Putting $m=m_1$ in \eqref{the 2nd condition induced in main THM1} and then subtracting the associated modification of \eqref{the 2nd condition induced in main THM1} from the equation 
\begin{equation}\label{a modification of the 4th condition induced in main THM1}
\Re\left[i\beta_1 z_2 P_{z_2}(z_2)\right]\equiv0
\end{equation}on $\triangle_{\epsilon_0}$ induced by \eqref{the 5th condition induced in main THM1}, we obtain 
\[
\Re\left[\left(\frac{1+P^{2}(z_2)}{2}\right)ia_{m_1+1,0}{(i-P(z_2))}^{m_1}+i^{m_1}c_{k_0}(z^{k_0}_{2}+o({|z_2|}^{k_0}))P_{z_2}(z_2)\right]\equiv0
\]on $\triangle_{\epsilon_0}$, which contradicts to Proposition~\ref{technical lemma1:THM1}. Hence we have $b_{m_1}(z_2)=i^{1-m_1}\beta_1 z_2$ for some $\beta_1\in\mathbb R^{*}$. Substituting this into \eqref{the 2nd condition induced in main THM1}, one gets
\begin{equation}\label{an induced condition by the 2nd condition induced in main THM1}
\Re\left[\left(\frac{1+P^{2}(z_2)}{2}\right)ia_{m_1+1,0}{(i-P(z_2))}^{m_1}+i^{1-m_1}\beta_1 z_2 P_{z_2}(z_2){(i-P(z_2))}^{m_1}\right]\equiv0
\end{equation}on $\triangle_{\epsilon_0}$. Subtracting \eqref{an induced condition by the 2nd condition induced in main THM1} from \eqref{a modification of the 4th condition induced in main THM1}, we have
\[
\Re\left[\left(\frac{1+P^{2}(z_2)}{2}\right)ia_{m_1+1,0}{(i-P(z_2))}^{m_1}-m_1\beta_1(z_2+o(|z_2|))P_{z_2}(z_2)P(z_2)\right]\equiv0
\]on $\triangle_{\epsilon_0}$, which again contradicts to Proposition~\ref{technical lemma1:THM1}.

Altogether, in this case, we obtain $h_{2}(z_1,z_2)\equiv i\beta z_2$ and $P(z_2)\equiv P(|z_2|)$ for some $\beta\in\mathbb R^{*}$ on $\triangle_{\epsilon_0}$.

\medskip

\noindent{\bf Case~2.} $m_0\geq1$. In this case, by Proposition~\ref{technical lemma1:THM1}, we first obtain $n_0=1$ and $b_{m_0,1}=i^{1-m_0}\beta z_2$ for some $\beta\in\mathbb R^{*}$. Then, by a change of variables, we may assume that
\[
b_{m_0}(z_2)=\sum_{n=0}^{\infty}b_{m_0,n}z^{n}_{2}=i^{1-m_0}\beta z_2.
\]Therefore, in this case, \eqref{the 4th condition induced in main THM1} can be re-written as
\begin{equation}\label{a modification of the 4th condition in main THM1 in Case2}
\Re\left[\left(\frac{1+P^{2}(z_2)}{2}\right)ia_{m_0+1,0}{(i-P(z_2))}^{m_0}+i^{1-m_0}\beta z_2P_{z_2}(z_2){(i-P(z_2))}^{m_0}\right]\equiv0
\end{equation}on $\triangle_{\epsilon_0}$.

We now divide the argument into two subcases as follows.

\medskip

\noindent{\bf Subcase~2.1.} $a_{m_0+1,0}=0$. In this subcase, it follows from \eqref{a modification of the 4th condition in main THM1 in Case2} that
\begin{equation}\label{R1-Subcase2.1 in main THM1}
\Re\left[i^{1-m_0}\beta z_2 P_{z_2}(z_2){(i-P(z_2))}^{m_0}\right]\equiv0
\end{equation}on $\triangle_{\epsilon_0}$.

Let $r\in(0,\epsilon_0)$ such that $P(r)\neq0$. Then we let $\gamma: [t_0,+\infty)\rightarrow\mathbb C$ be a curve such that $\gamma'(t)=i^{1-m_0}\beta\gamma(t){(i-P(\gamma(t)))}^{m_0}$ and $\gamma(t_0)=r$. Then setting $u(t)=P(\gamma(t))$, \eqref{R1-Subcase2.1 in main THM1} shows that $u'(t)\equiv0$, and hence $u(t)\equiv P(r)$. Therefore, we have
\[
\gamma'(t)=a\gamma(t);\;\gamma(t_0)=r,
\]where $a:=i^{1-m_0}\beta{(i-P(r))}^{m_0}$. This yields $\gamma(t)=r\exp(a(t-t_0))$. Since $|\gamma(t)|=r\exp((\Re\;a)(t-t_0))$ and $\gamma(t_0)=r\neq0$, we momentarily assume that $\Re\;a<0$. Since $0\leq|\gamma(t)|=r\exp((\Re\;a)(t-t_0))$, we get $\gamma(t)\rightarrow0$ as $t\rightarrow\infty$, and hence $P(r)\equiv P(\gamma(t))\rightarrow P(0)=0$, which contradicts to our choice of $r\in(0,\epsilon_0)$. In the case when $\Re\;a>0$, one can proceed the same argument as above~(by considering a curve $\tilde\gamma:(-\infty,t_0]\rightarrow\mathbb C$ instead of the above curve $\gamma$). 

\medskip

\noindent{\bf Subcase~2.2.} $a_{m_0+1,0}\neq0$. In this subcase, it follows from \eqref{a modification of the 4th condition in main THM1 in Case2} that on $\triangle_{\epsilon_0}$
\begin{equation}\label{R2-Subcase2.2 in main THM1}
\Re\left[i^{1-m_0}\beta z_2 P_{z_2}(z_2){(i-P(z_2))}^{m_0}\right]\equiv(\delta+\epsilon(z_2))P(z_2),
\end{equation}where $\delta:=\Re(m_0i^{m_0}a_{m_0+1,0}/2)\in\mathbb R^{*}$ and $\epsilon: \triangle_{\epsilon_0}\rightarrow\mathbb R$ is a smooth function with the condition that $\epsilon(z_2)\rightarrow0$ as $z_2\rightarrow0$. Without loss of generality, we may assume that $\delta<0$ and $|\epsilon(z_2)|<|\delta|/2$ on $\triangle_{\epsilon_0}$.

Let $r\in(0,\epsilon_0)$ such that $P(r)\neq0$. Then we let $\gamma: [t_0,+\infty)\rightarrow\mathbb C$ such that $\gamma'(t)=i^{1-m_0}\beta\gamma(t){(i-P(\gamma(t)))}^{m_0}$ and $\gamma(t_0)=r$. Then setting $u(t)=\frac{1}{2}\log|P(\gamma(t))|$, \eqref{R2-Subcase2.2 in main THM1} shows that $u'(t)=\delta+\epsilon(\gamma(t))$. Hence, we get
\begin{equation}\label{R3-Subcase2.2 in main THM1}
u(t)-u(t_0)=\delta(t-t_0)+\int_{t_0}^{t}\epsilon(\gamma(\tau))d\tau,\;\forall t\geq t_0.
\end{equation}This implies that $u(t)\rightarrow-\infty$ as $t\rightarrow+\infty$, and hence $\gamma(t)\rightarrow0$ as $t\rightarrow+\infty$. Moreover, since $|\epsilon(z_2)|<|\delta|/2$ on $\triangle_{\epsilon_0}$, it follows from \eqref{R3-Subcase2.2 in main THM1} that
\[
u(t)<u(t_0)+\frac{\delta}{2}(t-t_0),\;\forall t>t_0.
\]This inequality yields
\[
|P(\gamma(t))|\lesssim \exp(\delta t),\;\forall t>t_0. 
\]Therefore, $\gamma(t)$ satisfies the following:
\[
\gamma'(t)=\gamma(t)(i\beta+g(t)),
\]where $g: (t_0,+\infty)\rightarrow\mathbb C$ is a smooth function satisfying that $|g(t)|\lesssim\exp(\delta t)$. Then this yields
\[
\gamma(t)=r\exp\left(i\beta (t-t_0)+\int_{t_0}^{t}O(\exp(\delta \tau))d\tau\right);
\]hence $\gamma(t)\nrightarrow0$ as $t\rightarrow+\infty$, which contradicts to the discussion right after \eqref{R3-Subcase2.2 in main THM1}.

Hence, all the possible cases for the choice of $h_{2}$ are considered.

\medskip

Now we shall show that $h_{1}$ has the form of 
\[
h_{1}(z_1,z_2)=\alpha z_1,\;\alpha\in\mathbb R,
\]if $P(z_2)\equiv P(|z_2|)$ and $h_{2}(z_1,z_2)=i\beta z_2$ for some $\beta\in\mathbb R^{*}$. Suppose otherwise. Then there exists $j_0\geq2$ such that $h_{1}(z_1,z_2)=\alpha z_1+a_{j_0,0}z^{j_0}_{1}+o({|z_1|}^{j_0})$ with $a_{j_{0},0}\neq0$. Combining \eqref{re-written2 tangency condition in main THM1} with \eqref{conditions on a in main THM1}, $a_{0,0}=0$, and $P(z_2)\equiv P(|z_2|)$, we have
\begin{equation}\label{Discriminant-1 in main THM1}
\begin{split}
&\Re\left[\left(\frac{1}{2}+\frac{1}{2i}P(z_2)\right)\sum_{j=1}^{\infty}a_{j,0}{(it-tP(z_2))}^{j}+it\beta z_2 P_{z_2}(z_2)\right]\\
&=\Re\left[\left(\frac{1}{2}+\frac{1}{2i}P(z_2)\right)\sum_{j=1}^{\infty}a_{j,0}{(it-tP(z_2))}^{j}\right]\\
&=0
\end{split}
\end{equation}for all $z_2\in\mathbb C$ and $t\in\mathbb R$ with $z_2\in\triangle_{\epsilon_0}$ and $|t|<\delta_0$, where $\epsilon_0,\delta_0>0$ are small enough. Considering the coefficient of $t^{j}$ in \eqref{Discriminant-1 in main THM1}, for each $j\in\mathbb N^{*}$, we get 
\begin{equation}\label{Discriminant-2 in main THM1}
\Re\left[\left(\frac{1}{2}+\frac{1}{2i}P(z_2)\right)a_{j,0}{(i-P(z_2))}^{j}\right]\equiv0
\end{equation}on $\triangle_{\epsilon_0}$. Then one may regard \eqref{Discriminant-2 in main THM1} as an equation with a variable $P(z_2)$. For this reason, considering the coefficient of degree $1$ with respect to the variable $P(z_2)$ in \eqref{Discriminant-2 in main THM1}, we have
\begin{equation*}
\Re\left[-\frac{1}{2}(j-1){i}^{j-1}a_{j,0}\right]=0
\end{equation*}for each $j\in\mathbb N^{*}$. This conjunction with \eqref{conditions on a in main THM1} yields
\begin{equation*}
\Re\left(ia_{1,0}\right)=0;\;a_{s,0}=0,\;\forall s\geq2.
\end{equation*}Moreover, we note that if $P$ is rotationally symmetric, then
\[
H:=\alpha z_1\frac{\partial}{\partial z_1}+i\beta z_2\frac{\partial}{\partial z_2},
\]where $\alpha\in\mathbb R$ and $\beta\in\mathbb R^{*}$, always satisfies the condition \eqref{tangency condition in main THM1}. Hence we complete the proof.

\section{Proofs of Theorems~\ref{main THM2} and \ref{main THM3}}\label{section for the proofs of main THM2 and main THM3}
In this section, we continue the study of a $1$-nonminimal infinite type model $(M_P,0)$. As mentioned above, Theorems~\ref{main THM2} and \ref{main THM3} present the investigation of the associated holomorphic vector fields under certain conditions of local defining functions. For the proofs of these main theorems, we now prepare the following two technical lemmas. For the sake of brevity, we omit the proofs (see \cite{HN2016} for the details of the proofs).

\begin{lemma}[{\cite[Lemma~$1$]{HN2016}}]\label{modulus g' lemma}
Let $P:\triangle_{\epsilon_0}\rightarrow\mathbb R$ be a $C^{\infty}$-smooth function satisfying $\nu_{0}(P)=+\infty$ and $P(z)\nequiv0$. Suppose that there exists a conformal map $g$ on $\triangle_{\epsilon_0}$ with $g(0)=0$ such that 
\[
P(g(z))=(\beta+o(1))P(z),\quad z\in\triangle_{\epsilon_0},
\]for some $\beta\in\mathbb R^{*}$. Then $|g'(0)|=1$.
\end{lemma}

\begin{lemma}[{\cite[Lemma~$3$]{HN2016}}]
Let $P$ be a non-zero $C^\infty$-smooth function with $P(0)=0$ and let $g$ be a conformal map satisfying $g(0)=0, |g'(0)|=1$, and $g\neq\mathrm{id}$. If there exists a real number $\delta\in\mathbb R^{*}$ such that $P(g(z))\equiv\delta P(z)$, then $\delta=1$. Moreover, we have either $g'(0)=\exp(2\pi i p/q)\;(p,q\in\mathbb Z)$ and $g^{q}=\mathrm{id}$ or $g'(0)=\exp(2\pi i\theta)$ for some $\theta\in\mathbb R\backslash\mathbb Q$.
\end{lemma}

\subsection{Proof of Theorem~\ref{main THM2}}
%\begin{theorem}\label{main THM2}
%Let $(M_P,0)$ be a $C^\infty$-smooth hypersurface in $\mathbb C^2$ defined by the equation $\rho(z)=\rho(z_1,z_2):=\Re\;z_1+(\Im\;z_1)P(z_2)=0$, where $P$ is a $C^\infty$-smooth function on a neighborhood of $0$, vanishing to infinite order at $z_2=0$, and satisfying:
%\begin{enumerate}
%\item[$(i)$] $P(z_2)\notequiv0$ on a neighborhood of $z_2=0$;
%\item[$(ii)$] the connected component of $0$ in $S_{\infty}(P)$ is $\{0\}$.
%\end{enumerate}Then the following assertions hold:
%\begin{enumerate}
%\item[$(a)$] $\mathfrak{aut}(M_{P},0)=\mathfrak{aut}_{0}(M_{P},0)$.
%\item[$(b)$] If $\mathfrak{aut}_{0}(M_{P},0)=\{\alpha z_1\partial z_1:\alpha\in\mathbb R\}$, then
%\[
%\mathrm{Aut}(M_{P},0)=G_{2}(M_{P},0),
%\]where $G_{2}(M_{P},0)$ is the set of all $\mathrm{CR}$ automorphisms of $M_{P}$ defined by
%\[
%(z_1,z_2)\mapsto (Cz_1,g_{2}(z_2))
%\]for some constant $C\in\mathbb R^{*}$ and some holomorphic function $g_2$ with $g_{2}(0)=0$ and $|g'_{2}(0)|=1$ defined on a neighborhood of the origin in $\mathbb C$ satisfying that $P(g_2(z_2))\equiv P(z_2)$.
%\end{enumerate}
%\end{theorem}

%\begin{proof}
$(a)$ Let $H=h_{1}(z_1,z_2)\partial_{z_1}+h_{2}(z_1,z_2)\partial_{z_2}\in\mathfrak{aut}(M_P,0)$ be arbitrary. That is, $H$ is a holomorphic vector field near the origin in $\mathbb C^2$ such that 
\[(\Re\;H)\rho(z)=0
\]for all $z\in M_{P}$. We assume that $\{\phi_{t}\}_{t\in\mathbb R}\subset\mathrm{Aut}(M_P,0)$ is the associated subgroup generated by $H$. Since $\phi_{t}$ is biholomorphic for every $t\in\mathbb R$, the set $\{\phi_{t}(0,0): t\in\mathbb R\}$ is contained in $P_{\infty}(M_P)$. Moreover, since the connected component of $(0,0)$ in $P_{\infty}(M_{P})$ is $\{(is,0): s\in\mathbb R\}$, one gets $\phi_{t}(0,0)\in\{(is,0): s\in\mathbb R\}$ for every $t\in\mathbb R$. This relation yields
\begin{equation}\label{R1 in main THM2}
\Re\;h_{1}(0,0)=h_{2}(0,0)=0.
\end{equation}Then we immediately prove the assertion~$(a)$, if $\Im\;h_{1}(0,0)=0$.

For this reason, we shall now consider the case when $\Im\;h_{1}(0,0)\neq0$. Expanding the functions $h_1$ and $h_2$ into the Taylor series at the origin, 
\begin{align*}
h_{1}(z_1,z_2)&=\sum_{j,k=0}^{\infty}a_{j,k}z^{j}_{1}z^{k}_{2};\\
h_{2}(z_1,z_2)&=\sum_{j,k=0}^{\infty}b_{j,k}z^{j}_{1}z^{k}_{2},
\end{align*}where $a_{j,k}, b_{j,k}\in\mathbb C$ as in the proof of Theorem~\ref{main THM1}. Since $P(z_2)$ vanishes to infinite order at $z_2=0$, Eq.~\eqref{the first induced condition in main THM1} in the proof of Theorem~\ref{main THM1} yields
\begin{equation}\label{R2 in main THM2}
\Re\;a_{0,0}=0;\;a_{0,\ell}=0,\;\forall\ell\in\mathbb N^{*}.
\end{equation}If $\Im\;a_{0,0}\neq0$, then \eqref{the first induced condition in main THM1} and \eqref{R2 in main THM2} imply that
\begin{equation*}
\begin{split}
\Re\left[\left(\frac{1}{2}+\frac{P(z_2)}{2i}\right)\sum_{k=0}^{\infty}a_{0,k}z^{k}_{2}\right]&=\Re\left[\left(\frac{1}{2}+\frac{P(z_2)}{2i}\right)i(\Im\;a_{0,0})\right]\\
&=\frac{(\Im\;a_{0,0})}{2}P(z_2)\\
&\equiv0
\end{split}
\end{equation*}on $\triangle_{\epsilon_0}$; hence, $P(z_2)\equiv0$ on $\triangle_{\epsilon_0}$ which contradicts to our assumption $(i)$. Combining this fact with \eqref{R1 in main THM2}, we obtain the vanishing property of $H$ at the origin in $\mathbb C^2$.

In addition, by the definition of $\mathfrak{aut}_{0}(M_P,0)$, it is clear that 
\[
\mathfrak{aut}_{0}(M_P,0)\subset\mathfrak{aut}(M_P,0).
\]This completes the proof of $(a)$. 

\medskip

\noindent$(b)$ We first assume that 
\[
\mathfrak{aut}_{0}(M_P,0)=\{\alpha z_1\partial_{z_1}: \alpha\in\mathbb R\}.
\]Then it follows from the assertion~$(a)$ that 
\[
\mathfrak{aut}(M_P,0)=\{\alpha z_1\partial_{z_1}:\alpha\in\mathbb R\}
\]also holds.

Now let us denote by $\{T_{t}\}_{t\in\mathbb R}$ the $1$-parameter subgroup generated by $z_1\partial z_1$, that is, 
\[
T_{t}(z_1,z_2)=(\exp(t)z_1,z_2),\;t\in\mathbb R. 
\]For any $f=(f_1,f_2)\in\mathrm{Aut}(M_P,0)$, we define a family $\{F_t\}_{t\in\mathbb R}$ of automorphisms by setting
\[
F_{t}:=f\circ T_{-t}\circ f^{-1}.
\]Then it follows that $\{F_{t}\}_{t\in\mathbb R}$ is a $1$-parameter subgroup of $\mathrm{Aut}(M_P)$. Moreover, since $\mathfrak{aut}(M_P,0)=\{\alpha z_1\partial_{z_1}: \alpha\in\mathbb R\}$, the holomorphic vector field $H$ generated by $\{F_{t}\}_{t\in\mathbb R}$ belongs to $\{\alpha z_1\partial_{z_1}:\alpha\in\mathbb R\}$. This means that there exists a real number $\delta$ such that
\[
H=\delta z_1\partial_{z_1},
\]which yields
\[
F_{t}(z_1,z_2)=(\exp(\delta t)z_1,z_2),\;t\in\mathbb R.
\]This implies that for $t\in\mathbb R$
\[
f=T_{\delta t}\circ f\circ T_{t}
\]which is equivalent to
\begin{align}
f_{1}(z_1,z_2)&=\exp(\delta t)f_{1}(\exp(t)z_1,z_2); \label{R1-(b) in main THM2}\\
f_{2}(z_1,z_2)&=f_{2}(\exp(t)z_1,z_2). \label{R2-(b) in main THM2}
\end{align}Taking the derivative of both sides of \eqref{R1-(b) in main THM2} with respect to $t$, we have
\[
0=\delta\exp(\delta t)f_{1}(\exp(t)z_1,z_2)+\exp(\delta t)\exp(t)z_1\dfrac{\partial f_1(\exp(t)z_1,z_2)}{\partial(\exp(t)z_1)}.
\]This relation yields
\[
0=\delta f_1(z_1,z_2)+z_{1}\frac{\partial f_1}{\partial z_1}(z_1,z_2);
\]hence one can deduce that
\[
f_{1}(z_1,z_2)=z^{-\delta}_{1}g_{1}(z_2),
\]where $g_{1}$ is a holomorphic function on a neighborhood of $z_2=0$. Moreover, since $f_1$ is a biholomorphism, the constant $\delta$ should be $-1$.

Applying the same procedure as above to \eqref{R2-(b) in main THM2}, one can also deduce that
\[
f_{2}(z_1,z_2)=g_{2}(z_2),
\]where $g_2$ is a holomorphic function on a neighborhood of $z_2=0$ with $g_{2}(0)=0$.

Now we shall determine $f$ more precisely. Since $(it-tP(z_2),z_2)\in M_{P}$, $t\in\mathbb R$, and $M_P$ is invariant under $f$, we get
\begin{equation}\label{R4-(b) in main THM2}
\begin{split}
0&=\Re\left(f_{1}(it-tP(z_2),z_2)\right)+\Im\left(f_{1}(it-tP(z_2),z_2)\right)P(f_{2}(it-tP(z_2),z_2))\\
&=\Re\left((it-tP(z_2))g_{1}(z_2)\right)+\Im\left((it-tP(z_2))g_{1}(z_2)\right)P(g_{2}(z_2)).
\end{split}
\end{equation}Since the case $g_1\equiv0$ contradicts to the fact that $f$ is biholomorphic near the origin, we may assume that $g_{1}\nequiv0$. Then \eqref{R4-(b) in main THM2} implies that
\begin{equation}\label{R5-(b) in maim THM2}
P(g_{2}(z_2))=-\dfrac{\Re\left(g_{1}(z_2)(i-P(z_2))\right)}{\Im\left(g_{1}(z_2)(i-P(z_2))\right)}
\end{equation}for sufficiently small $|z_2|\in\mathbb R$. Since $P(g_{2}(z_2))$ vanishes to infinite order at $z_2=0$, $\Re\left(g_{1}(z_2)(i-P(z_2))\right)$ also has the same property at $z_2=0$. Moreover, by the same reason, we can further say that $\Re(ig_{1}(z_2))$ vanishes to infinite order at $z_2=0$. Combining this with the fact that $g_1$ is holomorphic near $z_2=0$, we obtain
\[
g_{1}(z_2)\equiv\;\text{a constant}\;C\in\mathbb R^{*}.
\]Therefore, \eqref{R5-(b) in maim THM2} can be re-written as
\[
P(g_{2}(z_2))\equiv P(z_2).
\]near the origin. Applying Lemma~\ref{modulus g' lemma} to this relation, we also obtain $|g_{2}'(0)|=1$ which finishes the proof of $(b)$.

Altogether, we complete the proof of Theorem~\ref{main THM2}.

%\end{proof}

\subsection{Proof of Theorem~\ref{main THM3}}
$(a)$ Let $H=h_{1}(z_1,z_2)\partial_{z_1}+h_{2}(z_1,z_2)\partial_{z_2}\in\mathfrak{aut}_{0}(M_P,0)$ be arbitrary. That is, $H$ is a holomorphic vector field near the origin such that 
\[
H(0,0)=0;\;(\Re\;H)\rho(z)=0
\]for all $z\in M_{P}$. Then we define a holomorphic vector field $\widetilde H$ by setting
\[
\widetilde H:=H-\alpha z_1\partial_{z_1},\;\alpha\in\mathbb R.
\]

Now we expand the functions $h_1-\alpha z_1$ and $h_2$ into the Taylor series at the origin:
\begin{align*}
h_{1}(z_1,z_2)-\alpha z_1&=\sum_{j,k=0}^{\infty}a_{j,k}z^{j}_{1}z^{k}_{2};\\
h_{2}(z_1,z_2)&=\sum_{j,k=0}^{\infty}b_{j,k}z^{j}_{1}z^{k}_{2},
\end{align*}where $a_{j,k}, b_{j,k}\in\mathbb C$. Then it follows from $\widetilde H(0,0)=0$ that 
\[
a_{0,0}=b_{0,0}=0.
\]Moreover, since $(it-tP(z_2),z_2)\in M_{P}$ with a small enough $t\in\mathbb R$, the tangency condition for $\widetilde H$ can be written as 
\begin{equation*}
\begin{split}
&(\Re\;\widetilde H)\rho(z)\\
&=\Re\left[\left(\frac{1}{2}+\frac{P(z_2)}{2i}\right)\sum_{j,k=0}^{\infty}a_{j,k}{(it-tP(z_2))}^{j}z^{k}_{2}+tP_{z_2}(z_2)\sum_{m,n=0}^{\infty}b_{m,n}{(it-tP(z_2))}^{m}z^{n}_{2}\right]\\
&=0
\end{split}
\end{equation*}for all $z_{2}\in\mathbb C$ and $t\in\mathbb R$ with $z_2\in\triangle_{\epsilon_{0}}$ and $|t|<\delta_0$, where $\epsilon_{0}, \delta_{0}>0$ are small enough.

Applying the same argument as in the proof of Theorem~\ref{main THM1}, one can obtain the following: for all $m\in\mathbb N^{0}$ and $\ell,\ell'\in\mathbb N^{*}$,
\begin{equation}\label{R1-(a) in main THM3}
\begin{split}
&a_{0,m}=0;\\
&\Re\left[\left(\left(\frac{1+P^{2}(z_2)}{2}\right)\sum_{k=0}^{\infty}ia_{m+1,k}z^{k}_{2}+P_{z_2}(z_2)\sum_{n=0}^{\infty}b_{m,n}z^{n}_{2}\right){(i-P(z_2))}^{m}\right]\equiv0\;\text{on}\;\triangle_{\epsilon_{0}};\;\\
&\Re\left(i^{\ell}a_{\ell,0}\right)=0;\\
&a_{\ell,\ell'}=0.
\end{split}
\end{equation}With these observations, we note that the coefficients $a_{\ell,0}, \ell\in\mathbb N^{*}$, only can be candidates to be non-zero among all the coefficients $a_{j,k}$.

Now we shall show that $\widetilde H\equiv0$. Aiming for a contradiction, we suppose that $\widetilde H\nequiv0$. Since $P(z_2)$ and $P_{z_2}(z_2)$ vanish to infinite order at $z_2=0$, one can see that if $h_2\equiv0$, then the above tangency condition yields $h_{1}(z_1,z_2)=\alpha z_1, \alpha\in\mathbb R$. Therefore, in the remaining of the proof, we focus our attention only on the case when $h_{2}\nequiv0$.

We shall divide our argument into the following two cases.

\medskip

\noindent{\bf Case~1.} $h_1(z_1,z_2)-\alpha z_1\nequiv0$. In this case, let $m_0$ be the smallest integer such that $b_{m_0,n}\neq0$ for some integer $n$, and then let $n_0$ be the smallest integer such that $b_{m_0,n_0}\neq0$. Since $h_{2}(0,0)=b_{0,0}=0$, we first observe that $m_0\geq1$ if $n_0=0$. For such fixed $m_0$ and $n_0$, \eqref{R1-(a) in main THM3} yields 
\begin{equation}\label{R2-(a) in main THM3}
\Re\left[\left(\left(\frac{1+P^{2}(z_2)}{2}\right)ia_{m_0+1,0}+P_{z_2}(z_2)b_{m_0,n_0}(z^{n_0}_{2}+o({|z_2|}^{n_0}))\right){(i-P(z_2))}^{m_0}\right]\equiv0
\end{equation}on $\triangle_{\epsilon_0}$. Moreover, we remark that $P_{z_2}(z_2)=\frac{1}{2}{\widetilde P}'(x)$, where $x:=\Re(z_2)$. In addition, if $b_{m_0,n_0}\neq0$, then we get
\[
\Re\left[P_{z_2}(z_2)b_{m_0,n_0}{(i-P(z_2))}^{m_0}(z^{n_0}_{2}+o({|z_2|}^{n_0}))\right]\neq0
\]on $\triangle_{\epsilon_0}$. Indeed, if $\Re\left[P_{z_2}(z_2)b_{m_0,n_0}{(i-P(z_2))}^{m_0}(z^{n_0}_{2}+o({|z_2|}^{n_0}))\right]=0$, then the binomial theorem shows that $b_{m_0,n_0}=0$ since ${\widetilde P}^{'}(x)\nequiv0$ near $x=0$ and the functions $P(z_2)$ and $P_{z_2}(z_2)$ vanish to infinite order at $z_2=0$. This contradicts to the choice of the pair $(m_0,n_0)$ such that $b_{m_0,n_0}\neq0$. Then it follows from \eqref{R2-(a) in main THM3} that
\begin{equation}\label{R3-(a) in main THM3}
{\widetilde P}^{'}(x)=-\dfrac{\Re\left[\left(1+P^{2}(z_2)\right)i{(i-P(z_2))}^{m_0}a_{m_0+1,0}\right]}{\Re\left[b_{m_0,n_0}{(i-P(z_2))}^{m_0}(z^{n_0}_{2}+o({|z_2|}^{n_0}))\right]},
\end{equation}for all $z_2:=x+iy\in\triangle_{\epsilon_0}$ satisfying
\[
{\widetilde P}^{'}(x)\neq0;\;\Re\left[b_{m_0,n_0}{(i-P(z_2))}^{m_0}(z^{n_0}_{2}+o({|z_2|}^{n_0}))\right]\neq0.
\]If $n_0\geq1$, then the right-hand side of \eqref{R3-(a) in main THM3} depends on $x$ and $y$; however, the left-hand side of \eqref{R3-(a) in main THM3} is independent of $y$ which leads to a contradiction.

In addition, if $n_0=0$, then \eqref{R3-(a) in main THM3} yields
\[
{\widetilde P}^{'}(x)\equiv-\dfrac{\Re\left[\left(1+P^{2}(z_2)\right)i{(i-P(z_2))}^{m_0}a_{m_0+1,0}\right]}{\Re\left[b_{m_0,0}{(i-P(z_2))}^{m_0}(1+o(1))\right]}\sim\widetilde P^2(x)
\]near the origin. This implies that $\frac{{\widetilde P}^{'}(x)}{\widetilde P^2(x)}$ becomes a bounded function near the origin. Integrating $\dfrac{{\widetilde P}^{'}(x)}{\widetilde P^{2}(x)}$, we get
\[
-\dfrac{1}{\widetilde P(x)}+\dfrac{1}{\widetilde P(x_0)}=\int_{x_0}^{x}\dfrac{\widetilde P'(t)}{\widetilde P^{2}(t)}dt,
\]where $x$ and $x_0$ are in a neighborhood of the origin. In this case, we obtain $\widetilde P(x)\nrightarrow0$ as $x\rightarrow0$, which is absurd.

\medskip

\noindent{\bf Case~2.} $h_1(z_1,z_2)-\alpha z_1\equiv0$. Let $m_0$ and $n_0$ be as in Case~$1$. Since $P(z_2)=o({|z_2|}^{\ell})$ for any $\ell\in\mathbb N^{*}$, \eqref{R1-(a) in main THM3} implies that
\[
\frac{1}{2}{\widetilde P}^{'}(x)\Re\left[b_{m_0,n_0}{(i-P(z_2))}^{m_0}(z^{n_0}_{2}+o({|z_2|}^{n_0}))\right]=0
\]for all $z_2:=x+iy\in\triangle_{\epsilon_0}$. Moreover, since ${\widetilde P}^{'}(x)\nequiv0$ near the origin, we get 
\[\Re\left[b_{m_0,n_0}{(i-P(z_2))}^{m_0}(z^{n_0}_{2}+o({|z_2|}^{n_0}))\right]=0
\]for all $z_2\in\triangle_{\epsilon_0}$, which is absurd as we observed in the previous case.

Altogether, one can say that $\mathfrak{aut}_{0}(M_{P},0)=\{\alpha z_1\partial_{z_1}: \alpha\in\mathbb R\}$.

\medskip

\noindent Now it remains to show that
\[
\mathfrak{aut}(M_P,0)=\mathfrak{g}_{1}\oplus\mathfrak{g}_{0},
\]where $\mathfrak{g}_{1}=\{\alpha z_1\partial_{z_1}: \alpha\in\mathbb R\}$ and $\mathfrak{g}_{0}=\{i\beta\partial_{z_2}: \beta\in\mathbb R\}$.

In what follows, by abuse of notation, let $H=h_{1}(z_1,z_2)\partial_{z_1}+h_{2}(z_1,z_2)\partial_{z_2}$ stand for an arbitrary element of $\mathfrak{aut}(M_P,0)$ and then let $\{\phi_{t}\}_{t\in\mathbb R}\subset\mathrm{Aut}(M_P)$ be the $1$-parameter subgroup generated by the vector field $H$. Since $\phi_{t}$ is biholomorphic for every $t\in\mathbb R$, the set $\{\phi_{t}(0,0)\colon t\in\mathbb R\}$ is contained in $P_{\infty}(M_P)$.

Furthermore, since the connected component of $(0,0)$ in $P_\infty(M_{\widetilde  P})$ is $\{(it, 0)\in \mathbb C^2\colon t\in \mathbb R\}$, one can deduce that the connected component of $(0,0)$ in $P_{\infty}(M_P)$ is $\{(is,is')\in\mathbb C^2\colon s,s'\in\mathbb R\}$. Therefore, we have
\[
\phi_{t}(0,0)\subset\{(is,is')\in\mathbb C^2\colon s,s'\in\mathbb R\}.
\]This yields
\[
\Re\;h_{1}(0,0)=\Re\;h_{2}(0,0)=0.
\]Hence, the holomorphic vector field 
\[
H(z_1,z_2)-i\beta_1\partial_{z_1}-i\beta_2\partial_{z_2},
\]where $\beta_1:=\Im\;h_{1}(0,0)$ and $\beta_2:=\Im\;h_{2}(0,0)$, belongs to $\mathfrak{aut}_{0}(M_P,0)$. However, the tangency condition $-i\beta_1\partial_{z_1}\in\mathfrak{aut}(M_P,0)$ holds, only if $\beta_1=0$. This ends the proof of the assertion~$(a)$.

\medskip

\noindent $(b)$ By $(a)$, we see that $\mathfrak{aut}(M_P,0)=\mathfrak{g}_{1}\oplus\mathfrak{g}_{0}$, that is, $z_1\partial_{z_1}$ and $i\partial_{z_2}$ generate $\mathfrak{aut}(M_P,0)$.

Now let us denote by $\{T^{1}_{t}\}_{t\in\mathbb R}$ and $\{T^{2}_{t}\}_{t\in\mathbb R}$ the $1$-parameter subgroups generated by $z_1\partial_{z_1}$ and $i\partial_{z_2}$ respectively, that is, 
\[
T^{1}_{t}(z_1,z_2)=(\exp(t)z_1,z_2);\;T^{2}_{t}(z_1,z_2)=(z_1,z_2+it)
\]for $t\in\mathbb R$. For any $f=(f_1,f_2)\in\mathrm{Aut}(M_P,0)$, we define families $\{F^{j}_{t}\}_{t\in\mathbb R}$ of automorphisms by setting
\[
F^{j}_{t}:=f\circ T^{j}_{-t}\circ f^{-1}\quad(j=1,2).
\]Then it follows that $\{F^{j}_{t}\}_{t\in\mathbb R}$, $j=1,2$, are $1$-parameter subgroups of $\mathrm{Aut}(M_P)$.

Moreover, since $\mathfrak{aut}(M_P,0)=\mathfrak{g}_{1}\oplus\mathfrak{g}_{0}$, each holomorphic vector field $H^{j}$ generated by $\{F^{j}_{t}\}_{t\in\mathbb R}\;(j=1,2)$, surely belongs to $\mathfrak{g}_{1}\oplus\mathfrak{g}_{0}$. This means that there exist real numbers $\delta^{j}_{1}, \delta^{j}_{2}$, $j=1,2$, such that
\[
H^{j}=\delta^{j}_{1}z_1\partial_{z_1}+i\delta^{j}_{2}\partial_{z_2}\quad(j=1,2),
\]which yields
\[
F^{j}_{t}(z_1,z_2)=(\exp(\delta^{j}_{1}t)z_1,z_2+i\delta^{j}_{2}t)=T^{1}_{\delta^{j}_{1}t}\circ T^{2}_{\delta^{j}_{2}t}(z_1,z_2)
\]for $j=1,2$ and $t\in\mathbb R$. This implies that
\[
f=T^{1}_{\delta^{j}_{1}t}\circ T^{2}_{\delta^{j}_{2} t}\circ f\circ T^{j}_{t}\quad(j=1,2),
\]which is equivalent to
\begin{align}
f_{1}(z_1,z_2)&=\exp(\delta^{1}_{1}t)f_{1}(\exp(t)z_1,z_2); \label{R1-(b) in main THM3}\\
f_{2}(z_1,z_2)&=f_{2}(\exp(t)z_1,z_2)+i\delta^{1}_{2}t;\label{R2-(b) in main THM3}\\
f_{1}(z_1,z_2)&=\exp(\delta^{2}_{1}t)f_{1}(z_1,z_2+it); \label{R3-(b) in main THM3}\\
f_{2}(z_1,z_2)&=f_{2}(z_1,z_2+it)+i\delta^{2}_{2}t. \label{R4-(b) in main THM3}
\end{align}Taking the derivative of both sides of \eqref{R1-(b) in main THM3} with respect to $t$, we have
\[
0=\delta^{1}_{1}\exp(\delta^{1}_{1}t)f_{1}(\exp(t)z_1,z_2)+\exp(\delta^{1}_{1}t)\exp(t)z_1\dfrac{\partial f_{1}(\exp(t)z_1,z_2)}{\partial(\exp(t)z_1)}.
\]This implies that $0=\delta^{1}_{1}f_{1}(z_1,z_2)+z_1\frac{\partial f_1}{\partial z_1}(z_1,z_2)$; hence, one gets
\begin{equation}\label{R5-(b) in main THM3}
f_{1}(z_1,z_2)=z^{-\delta^{1}_{1}}_{1}g_{1}(z_2),
\end{equation}where $g_1$ is a holomorphic function on a neighborhood of $z_2=0$. Moreover, since $f_1$ is a biholomorphism, $\delta^{1}_{1}$ should be $-1$.

Now we apply the same procedure as above to \eqref{R3-(b) in main THM3}. Then we first get
\[
0=\delta^{2}_{1}\exp(\delta^{2}_{1}t)f_{1}(z_1,z_2+it)+\exp(\delta^{2}_{1}t)i\dfrac{\partial f_{1}(z_1,z_2+it)}{\partial(z_2+it)},
\]which yields
\begin{equation}\label{R6-(b) in main THM3}
0=\delta^{2}_{1}f_{1}(z_1,z_2)+i\frac{\partial f_1}{\partial z_2}(z_1,z_2).
\end{equation}Substituting \eqref{R5-(b) in main THM3} into \eqref{R6-(b) in main THM3}, we obtain
\begin{equation}\label{R7-(b) in main THM3}
0=\delta^{2}_{1}z_1 g_{1}(z_2)+i z_1\frac{dg_1}{dz_2}(z_2).
\end{equation}Then \eqref{R7-(b) in main THM3} tells us that $g_1$ has a form
\[
g_{1}(z_2)=C_1\exp(i\delta^{2}_{1}z_2),
\]where $C_1$ is a constant which will be determined more precisely later on.

Next, applying the same argument as above to \eqref{R2-(b) in main THM3} and \eqref{R4-(b) in main THM3} again, one can deduce that
\begin{align}
0&=z_1\frac{\partial f_2}{\partial z_1}(z_1,z_2)+i\delta^{1}_{2}; \label{R8-(b) in main THM3}\\
0&=i\frac{\partial f_2}{\partial z_2}(z_1,z_2)+i\delta^{2}_{2}. \label{R9-(b) in main THM3}
\end{align}It follows from \eqref{R9-(b) in main THM3} that
\begin{equation}\label{R10-(b) in main THM3}
f_{2}(z_1,z_2)=-\delta^{2}_{2}z_2+h_{1}(z_1),
\end{equation}where $h_1$ is a holomorphic function on a neighborhood of $z_1=0$, fixing the origin. Substituting \eqref{R10-(b) in main THM3} into \eqref{R8-(b) in main THM3}, we get
\[
0=z_1\frac{dh_{1}}{dz_1}(z_1)+i\delta^{1}_{2}.
\]This clearly forces that $h_1$ should be identically zero since $h_1$ is a biholomorphism fixing the origin in $\mathbb C$; hence $0=z_1\frac{dh_1}{dz_1}(z_1)+i\delta^{1}_{2}=i\delta^{1}_{2}$. Therefore, we obtain
\[
f_{2}(z_1,z_2)=-\delta^{2}_{2}z_2.
\]

Altogether, we have
\[
f(z_1,z_2)=(C_1 z_1\exp(i\delta^{2}_{1}z_2),-\delta^{2}_{2}z_2).
\]

Now we shall determine $f$ more precisely. Since $M_P$ is invariant under $f$, one can deduce that
\begin{equation}\label{R11-(b) in main THM3}
\begin{split}
0&=\Re\left(f_{1}(it-tP(z_2),z_2)\right)+\Im\left(f_{1}(it-tP(z_2),z_2)\right)P\left(f_{2}(it-tP(z_2),z_2)\right)\\
&=\Re\left((it-tP(z_2))g_{1}(z_2)\right)+\Im\left((it-tP(z_2))g_{1}(z_2)\right)P(-\delta^{2}_{2}z_2)
\end{split}
\end{equation}for sufficiently small $|z_2|, t\in\mathbb R$. Since the case $g_1\equiv0$ contradicts to the fact that $f$ is biholomorphic near the origin, we may assume that $g_1\nequiv0$; hence \eqref{R11-(b) in main THM3} implies that
\begin{equation}\label{R12-(b) in main THM3}
P(-\delta^{2}_{2}z_2)=-\dfrac{\Re\left((i-P(z_2))g_{1}(z_2)\right)}{\Im\left((i-P(z_2))g_{1}(z_2)\right)}
\end{equation}for sufficiently small $|z_2|\in\mathbb R$. Since $P(-\delta^{2}_{2}z_2)$ vanishes to infinite order at $z_2=0$, $\Re\left((i-P(z_2))g_{1}(z_2)\right)$ also has the same property at $z_2=0$. In addition, since $P(z_2)$ vanishes to infinite order at $z_2=0$, one can further say that $\Re(ig_{1}(z_2))$ vanishes to infinite order at $z_2=0$. Combining this with the fact that $g_1$ is holomorphic near $z_2=0$, we obtain
\[
g_{1}(z_2)=\;\text{a constant}\;C\in\mathbb R^{*}.
\]Note that this yields the constants 
\[
C=C_{1};\;\delta^{2}_{1}=0.
\]Therefore, \eqref{R12-(b) in main THM3} can be re-written as
\[
P(-\delta^{2}_{2}z_2)=P(z_2).
\]Then it follows from Lemma~\ref{modulus g' lemma} that $|\delta^{2}_{2}|=1$. Thus, since $\delta^{2}_{2}$ was chosen in $\mathbb R$, the only two cases appeared in the statement of this theorem can occur as desired.

\medskip

\noindent $(c)$ Now let $f\in\mathrm{Aut}(M_P)$ be arbitrary. Then $f(0,0)$ is of infinite type. It follows from the assumption $P_{\infty}(M_{P})=\{(is,is')\in\mathbb C^2\colon s,s'\in\mathbb R\}$ that
\[
f(0,0)=(is_{0},is'_{0})\;\text{for some}\;s_{0}, s'_{0}\in\mathbb R.
\]Then composing with the automorphism $T^{2}_{-s'_0}$ appeared in the proof of $(b)$ of this theorem, one can deduce that
\[
T^{2}_{-s'_{0}}\circ f(0,0)=(is_0,0).
\]Indeed, we have $s_{0}=0$: Suppose otherwise. Then 
\[
g=(g_1,g_2):=T^{2}_{-s'_0}\circ f
\]satisfies that $g(0,0)=(is_{0},0)$ for some $s_{0}\in\mathbb R^{*}$. Expanding the functions $g_1$ and $g_2$ into the Taylor series at the origin, we have
\begin{align*}
g_{1}(z_1,z_2)&:=\sum_{j,k=0}^{\infty}a_{j,k}z^{j}_{1}z^{k}_{2};\\
g_{2}(z_1,z_2)&:=\sum_{m,n=0}^{\infty}b_{m,n}z^{m}_{1}z^{n}_{2},
\end{align*}where $a_{j,k}, b_{m,n}\in\mathbb C$. Since $g(0,0)=(is_{0},0)$ for some $s_{0}\in\mathbb R^{*}$, we have 
\begin{equation}\label{additional observation-(c)1-proof of the third main theorem}
a_{0,0}=is_{0};\;b_{0,0}=0.
\end{equation}Considering the points $z_2=0$ in the relation 
\[
\mathrm{Re}\left[g_{1}(z_1,z_2)\right]+\mathrm{Im}\left[g_{1}(z_1,z_2)\right]P(g_2(z_1,z_2))=0
\]for all $(z_1,z_2)\in M_{P}$, we get
\begin{equation}\label{additional observation-(c)2-proof of the third main theorem}
a_{\ell,\ell'}=0,
\end{equation}for all $\ell\in\mathbb N^{0}$ and $\ell'\in\mathbb N^{*}$. Since $g=(g_1,g_2)$ is an automorphism, \eqref{additional observation-(c)1-proof of the third main theorem} and \eqref{additional observation-(c)2-proof of the third main theorem} imply that
\[
a_{1,0}\neq0;\;b_{0,1}\neq0.
\]Moreover, since $M_P$ is invariant under the mapping $g$, one can get 
\begin{equation}\label{additional observation-(c)3-proof of the third main theorem}
\mathrm{Re}\left[g_{1}(it-t\widetilde P(\mathrm{Re}\;z_2),z_2)\right]+
\mathrm{Im}\left[g_{1}(it-t\widetilde P(\mathrm{Re}\;z_2),z_2)\right]P\left(g_{2}(it-t\widetilde P(\mathrm{Re}\;z_2),z_2)\right)=0
\end{equation}for all $z_2\in\mathbb C$ and $t\in\mathbb R$ with $z_2\in\triangle_{\epsilon_0}$ and $|t|<\delta_0$, where $\epsilon_0, \delta_0$ are small enough. Putting $t=0$ into \eqref{additional observation-(c)3-proof of the third main theorem} and then using \eqref{additional observation-(c)1-proof of the third main theorem} and \eqref{additional observation-(c)2-proof of the third main theorem}, we obtain
\begin{equation}\label{additional observation-(c)4-proof of the third main theorem}
s_{0}\widetilde P\left(\mathrm{Re}\left(\sum_{n=1}^{\infty}b_{0,n}z^{n}_{2}\right)\right)\equiv0
\end{equation}on $z_2\in\triangle_{\epsilon_0}$. Since $s_0\neq0$, \eqref{additional observation-(c)4-proof of the third main theorem} yields
\[
\widetilde P\left(\mathrm{Re}\left(\sum_{n=1}^{\infty}b_{0,n}z^{n}_{2}\right)\right)\equiv0
\]on $z_2\in\triangle_{\epsilon_0}$. However, this is absurd, since $\widetilde P(x)\not \equiv 0$ on a neighborhood of $x=0$ in $\mathbb R$ by the assumption $\mathrm{(i)}$ and the function $h(z_2):=\sum_{n=1}^{\infty}b_{0,n}z^{n}_{2}$ is a local biholomorphism at $z_2=0$. This completes the assertion. Hence, we obtain
\[
T^{2}_{-s'_{0}}\circ f\in\mathrm{Aut}(M_P,0),
\]where $\mathrm{Aut}(M_P,0)$ is explicitly described in the proof of $(b)$ of this theorem. This completes the proof of $(c)$.

\medskip

\noindent Altogether, we finish the proof of Theorem~\ref{main THM3}.

%\end{proof}

\section{Examples}\label{two examples}

We begin with this section by demonstrating the fact that there exists a $1$-nonminimal infinite type model $(M_P,0)$ in $\mathbb C^2$ such that $P_\infty(M_P)\ne \left\{(it-tP(z_2),z_2)\colon t\in\mathbb R, z_2\in S_{\infty}(P)\right\}$ as follows.
\begin{example}\label{0-e.g.}
Fix $z_2^0\in \mathbb C^*, C\in \mathbb C$, and $t_0\in \mathbb R^*$. Then fix $r$ such that $0<r<|z_2^0|/4$. Let us denote by $\chi$ a non-negative $\mathcal{C}^\infty$-smooth cut-off function on $\mathbb C$ such that
\[  
  \chi(z)=\begin{cases}
1\;&\text{if}\;|z|<r,\\
0\;&\text{if}\;|z|>2r.
  \end{cases}
  \]
Denote by $P$ a $\mathcal{C}^\infty$-smooth function defined on $\mathbb C$ by setting
\[
P(z_2)=\chi(z_2)\exp\left(-1/{|z_2|^2}\right)+\chi(z_2-z_2^0)\Big(C-\dfrac{\mathrm{Re}(z_2-z_2^0)+C\;\mathrm{Im}(z_2-z_2^0) }{t_0+\mathrm{Im}(z_2-z_2^0)} \Big)
\]
Then, one can see that $P(z_2^0)=C$ and
\begin{align}\label{counterexample-eqtn}
P(z_2^0+t)-P(z_2^0)=-\dfrac{\mathrm{Re}(t)+\mathrm{Im}(t) P(z_2^0)}{t_0+\mathrm{Im}(t)}
\end{align}
for all $t\in \Delta_{r}$. Consequently, one has $\nu_{z^0_2}(P)<+\infty$ and hence $z^0_2\not \in S_\infty(P)$. 

Let us define
$\mathcal{E}=\{(z_1,z_2)\in \mathbb C^2\colon \rho(z):= \mathrm{Re}(z_1)+\mathrm{Im}(z_1) P(z_2)=0\}$. We now prove that $\mathcal{E}$ contains an analytic set passing the point $(it_0-t_0 P(z_2^0), z_2^0)$. Let $z_1(t)$, $t\in \Delta_r$, be any non-zero holomorphic function with $z_1(0)=0$. Then, let $\gamma\colon \Delta_r\to \mathbb C^2$ be a holomorphic curve defined by
$\gamma(t)=(it_0-t_0 P(z_2^0)+z_1(t), z_2^0+z_1(t))$. Using the relation \eqref{counterexample-eqtn}, one can deduce that
\begin{align*}
\rho\circ \gamma(t)&=\mathrm{Re}(it_0-t_0 P(z_2^0)+z_1(t))+\mathrm{Im}(it_0-t_0 P(z_2^0)+z_1(t)) P(z_2^0+z_1(t))\\
&=-t_0 P(z_2^0)+\mathrm{Re}(z_1(t))+(t_0+\mathrm{Im}(z_1(t))) P(z_2^0+z_1(t)) \\
&=-t_0 P(z_2^0)+\mathrm{Re}(z_1(t))+(t_0+\mathrm{Im}(z_1(t))) P(z_2^0))\\
&\quad +(t_0+\mathrm{Im}(z_1(t)))\left(P(z_2^0+z_1(t))-P(z_2^0)\right)\\
&=\mathrm{Re}(z_1(t))+\mathrm{Im}(z_1(t)) P(z_2^0))+(t_0+\mathrm{Im}(z_1(t)))\left( P(z_2^0+z_1(t))-P(z_2^0)\right)\\
&=0
\end{align*}for all $t\in \Delta_r$, and the assertion hence follows.

Since $\mathcal{E}$ contains an analytic set passing the point $(it_0-t_0 P(z_2^0), z_2^0)$, we have $\tau(\mathcal{E},(it_0-t_0P(z_2^0), z_2^0) )=+\infty$. Thus, since $z_2^0\not\in  S_{\infty}(P)$, we obtain
\[
P_\infty(\mathcal{E})\ne \left\{(it-tP(z_2),z_2)\colon t\in\mathbb R, z_2\in S_{\infty}(P)\right\}.
\]
\end{example}

Now we shall investigate several examples as analogues of those in \cite[Section~$6$]{HN2016}.
\begin{example}\label{first e.g.}
Consider the model $M_{P_1}$, where $P_1$ is defined by setting
\begin{equation*}
P_{1}(z):=\left\{
\begin{aligned}
\exp\left(-1/{|z|}^a\right)&\quad\text{if}\;z\neq0,\\
0&\quad\text{if}\;z=0,
\end{aligned}
\right.
\end{equation*}where $a>0$. Then it is easily seen that $M_{P_1}$ satisfies the assumptions of Theorem~\ref{main THM1} and Theorem~\ref{main THM2}. Since $P_1\nequiv0$ near the origin in $\mathbb C$ and $P_1$ is rotationally symmetric, Theorem~\ref{main THM1} and Theorem~\ref{main THM2}~$(a)$ show that 
\[
\mathfrak{aut}(M_{P_1},0)=\mathfrak{aut}_{0}(M_{P_1},0)=\{\alpha z_1\partial_{z_1}+i\beta z_2\partial_{z_2}: \alpha, \beta\in\mathbb R\}.
\]In addition, we obtain
\[
\mathrm{Aut}(M_{P_1},0)=\{(z_1,z_2)\mapsto(sz_1,\exp(it)z_2)\colon s\in\mathbb R^{*}, t\in\mathbb R\},
\]which is clear from Theorem~\ref{main THM2}~$(b)$.
\end{example}

\begin{example}
Consider the model $M_{P_2}$, where $P_2$ is defined by setting
\begin{equation*}
P_{2}(z):=\left\{
\begin{aligned}
\exp\left(-1/{|z|}^a+\Re\;z\right)&\quad\text{if}\;z\neq0,\\
0&\quad\text{if}\;z=0,
\end{aligned}
\right.
\end{equation*}where $a>0$.

In this case we first observe that, by definition, $M_{P_2}$ satisfies the assumptions of Theorem~\ref{main THM1} and Theorem~\ref{main THM2}. In contrast with the previous example, $P_2$ is not rotationally symmetric, but $P_2$ is also not identically zero near the origin in $\mathbb C$. Then Theorem~\ref{main THM1} and Theorem~\ref{main THM2}~$(a)$ imply that
\[
\mathfrak{aut}(M_{P_2},0)=\mathfrak{aut}_{0}(M_{P_2},0)=\{\alpha z_1\partial_{z_1}\colon \alpha\in\mathbb R\}.
\]In addition, it follows from Theorem~\ref{main THM2}~$(b)$ that
\[
\mathrm{Aut}(M_{P_2},0)=G_{2}(M_{P_2},0).
\]
\end{example}

\begin{example}
Consider the model $M_{P_3}$, where $P_3$ is defined by setting
\begin{equation*}
P_{3}(z):=\left\{
\begin{aligned}
\exp\left(-1/{|\Re\;z|}^a\right)&\quad\text{if}\;\Re\;z\neq0,\\
0&\quad\text{if}\;\Re\;z=0,
\end{aligned}
\right.
\end{equation*}where $a>0$. Let us define a function $\widetilde P(z)$ by setting
\[
\widetilde P(z):=P_{1}(z),
\]where $P_{1}$ is given in the above Example~\ref{first e.g.}. Then it is easy to check that $\widetilde P$ allows the assumption of Theorem~\ref{main THM3}, and $P_{3}(z):=\widetilde P(\Re\;z)$. Combining the discussion in Example~\ref{first e.g.} with Theorem~\ref{main THM3}, one can see that 
\begin{align*}
\mathfrak{aut}_{0}(M_{P_3},0)&=\{\alpha z_1\partial_{z_1}\colon \alpha\in\mathbb R\};\\
\mathfrak{aut}(M_{P_3},0)&=\{\alpha z_1\partial_{z_1}+i\beta\partial_{z_2}\colon \alpha,\beta\in\mathbb R\};\\
\mathrm{Aut}(M_{P_3},0)&=\{(z_1,z_2)\mapsto (sz_1,\pm z_2)\colon s\in\mathbb R^{*}\};\\
\mathrm{Aut}(M_{P_3})&=\{(z_1,z_2)\mapsto (sz_1,\pm z_2+it)\colon s\in\mathbb R^{*}, t\in\mathbb R\}.
\end{align*}

\end{example}

\medskip

\appendix
\section*{Appendix}
\addcontentsline{toc}{section}{Appendices}
%\section{Concluding Remark}
In this Appendix, we shall describe an analogue of Theorem~\ref{main THM1} for \emph{$m$-nonminimal infinite type models} with $m>1$ in $\mathbb C^2$. Let us consider a $C^\infty$-smooth hypersurface $(M_{P,m},0)$ with $m>1$ in $\mathbb C^2$ defined by 
\[
M_{P,m}:=\{(w,z)\in\mathbb C^2\colon \rho(w,z):=\Im\;w-{(\Re\;w)}^{m}P(z)=0\},
\]where $P(z)$ is a $C^\infty$-smooth function on a neighborhood of the origin in $\mathbb C$ satisfying the two above conditions $(i)$ and $(ii)$ in Theorem~\ref{main THM1}.

Let $H=h_{1}(w,z)\partial_w+h_2(w,z)\partial_z\in\mathfrak{aut}_{0}(M_{P,m},0)$ be arbitrary. That is, $H$ is a holomorphic vector field near the origin in $\mathbb C^2$ such that 
\[
(\Re\;H)\rho(w,z)=0;\;H(0,0)=0
\]for all $(w,z)\in M_{P,m}$. Expanding the functions $h_1$ and $h_2$ into the Taylor series at the origin, 
\begin{align*}
h_{1}(w,z)&=\sum_{j,k=0}^{\infty}a_{j,k}w^jz^k;\\
h_{2}(w,z)&=\sum_{\ell,n=0}^{\infty}b_{\ell,n}w^\ell z^n,
\end{align*}where $a_{j,k}, b_{\ell,n}\in\mathbb C$. 

Since $(t+it^{m}P(z),z)\in M_{P,m}$ with $t\in\mathbb R$ small enough, the above tangency condition admits the following form:
\begin{equation}\label{tangency condition for s-nonminimal}
\begin{split}
&\Re{\left[\left(\frac{1}{2i}-\frac{mt^{m-1}P(z)}{2}\right)\sum_{j,k=0}^{\infty}a_{j,k}{(t+it^m P(z))}^{j}z^k
-t^{m}P_{z}(z)\sum_{\ell,n=0}^{\infty}b_{\ell,n}{(t+it^m P(z))}^{\ell}z^n\right]}\\
&=0
\end{split}
\end{equation}for all $z\in\mathbb C$ and $t\in\mathbb R$ with $z\in\triangle_{\epsilon_{0}}$ and $|t|<\delta_0$, where $\epsilon_0,\delta_0>0$ are sufficiently small. Since $P(z)$ and $P_z(z)$ vanish to infinite order at $z=0$, it follows from \eqref{tangency condition for s-nonminimal} that
\begin{equation}\label{conditions on a in Appendix}
\left\{
\begin{aligned}
a_{0,0}&=0;\\
\Im\left(a_{s,0}\right)&=0,\;\forall s\in\mathbb N^{*};\\
a_{s',\ell'+1}&=0,\;\forall s',\ell'\in\mathbb N^{0}.
\end{aligned}
\right.
\end{equation}Then \eqref{tangency condition for s-nonminimal} can be re-written as
\begin{equation}\label{re-written t.c. for s-nonminimal}
\begin{split}
&\Re\left[\frac{1}{2}P(z)\sum_{j=0}^{\infty}(j-m)a_{j,0}t^{m+j-1}-P_{z}(z)\sum_{\ell,n=0}^{\infty}b_{\ell,n}z^n t^{m+\ell}\right.\\
&\quad\quad\left.-iP(z)P_{z}(z)\sum_{\ell'=1}^{\infty}\sum_{n=0}^{\infty}\ell' b_{\ell',n}z^n t^{2m+\ell'-1}+o(|P(z)|)\right]\\
&=0
\end{split}
\end{equation}for all $z\in\triangle_{\epsilon_{0}}$ and $t\in\mathbb R$ sufficiently small.

If $h_2\equiv0$, then after considering the coefficient of $t^{m+j-1}$ for each $j$ in \eqref{re-written t.c. for s-nonminimal}, one can deduce that $a_{j,0}=0$ for all $j\neq m$. Then we obtain
\[
H=a_{m,0}w^{m}\partial_w.
\]Let us denote by $\{\varphi_{t}\}_{t\in\mathbb R}:=\{(\varphi^1(t),\varphi^2(t))\}_{t\in\mathbb R}\subset\mathrm{Aut}(M_{P,m},0)$ the $1$-parameter subgroup generated by $w^{m}\partial_w$, that is, for $t\in\mathbb R$
\[
\frac{d\varphi^{1}}{dt}(t)={(\varphi^1(t))}^{m};\;\frac{d\varphi^2}{dt}(t)=0
\]with $(\varphi^1(0),\varphi^2(0))=(w,z)\in M_{P,m}$. On the other hand, since $m>1$, the solution $\varphi_{t}$ of this initial value problem is not invertible, hence $\{\varphi_{t}\}_{t\in\mathbb R}\nsubset\mathrm{Aut}(M_{P,m},0)$ which leads to a contradiction. Hence, if $h_2\equiv0$, then we must have $h_1\equiv0$. For this reason, in the remaining of the proof, we always assume that $h_2\nequiv0$ without loss of generality.

Let $m_0$ be the smallest integer such that $b_{m_0,n}\neq0$ for some $n\in\mathbb N^{0}$. Then we let $n_0$ be the smallest integer such that $b_{m_0,n_0}\neq0$. Since $b_{0,0}=0$, it is clear that $m_0\geq1$ if $n_0=0$. We shall divide the argument into the following three cases.

\noindent{\bf Case~1.} $0\leq m_0< m-1$. Considering the coefficient of $t^{m+m_0}$ in \eqref{re-written t.c. for s-nonminimal}, we get
\begin{equation}\label{refined tangency condition-Case1-Appendix-fixing the origin}
\Re\left[\frac{1}{2}P(z)(m_0+1-m)a_{m_0+1,0}-P_{z}(z)\sum_{n=0}^{\infty}b_{m_0,n}z^n\right]\equiv0
\end{equation}on $\triangle_{\epsilon_{0}}$. In this case, by Proposition~\ref{technical lemma1:THM1}, we obtain $n_0=1$ and $b_{m_0,1}=i\beta$ for some $\beta\in\mathbb R^{*}$. Then, by a change of variables~(cf.~\cite[Lemma~$1$]{N2013}), we may assume that
\[
b_{m_0}(z):=\sum_{n=0}^{\infty}b_{m_0,n}z^n=i\beta z.
\]

Let $r\in(0,\epsilon_0)$ be an arbitrary number such that $P(r)\neq0$ and then let $v(t):=P(r\exp(it))$ for all $t\in\mathbb R$. Combining these relations with the above condition \eqref{refined tangency condition-Case1-Appendix-fixing the origin}, one gets
\[
\dfrac{v'(t)}{v(t)}=(m_0+1-m)\frac{a_{m_0+1,0}}{\beta}.
\]Integrating this, we obtain
\[
v(t)=v(0)\exp\left((m_0+1-m)\frac{a_{m_0+1,0}}{\beta}t\right)
\]for all $t\in\mathbb R$. Here, without loss of generality, we may take $\frac{a_{m_0+1,0}}{\beta}$ as a positive number. Then as $t\rightarrow+\infty$, we get $v(t)\rightarrow0$ and hence $P(r\exp(it))\rightarrow P(0)=0$, which contradicts to our choice of $r\in(0,\epsilon_0)$~(consider the associated limit of $v(t)$ as $t\rightarrow-\infty$ if $\frac{a_{m_0+1,0}}{\beta}<0$).

\medskip

\noindent{\bf Case~2.} $m_0=m-1$. In this case, we first note that $m+m_0=2m-1$. Considering the coefficient of $t^{2m-1}$ in \eqref{re-written t.c. for s-nonminimal}, we have
\begin{equation}\label{refined tangency condition-Case2-Appendix-fixing the origin}
\Re\left[P_{z}(z)\sum_{n=0}^{\infty}b_{m_0,n}z^n\right]\equiv0
\end{equation}on $\triangle_{\epsilon_0}$. Applying \cite[Corollary~$4$]{HN2016} to this relation, one can obtain
\[
n_0=1;\;\Re\left(b_{m_0,n_0}\right)=\Re\left(b_{m_0,1}\right)=0.
\]Therefore, by a change of variables, we may assume that
\[
b_{m_0}(z):=\sum_{n=0}^{\infty}b_{m_0,n}z^n=i\tilde\beta z
\]for some $\tilde\beta\in\mathbb R^{*}$. Combining this with the above condition \eqref{refined tangency condition-Case2-Appendix-fixing the origin}, we get
\[
\Re\left[i\tilde\beta zP_{z}(z)\right]\equiv0
\]on $\triangle_{\epsilon_0}$. This implies that $P(z)\equiv P(|z|)$ on $\triangle_{\epsilon_0}$.

We now prove that $b_{\ell}(z):=\sum_{n=0}^{\infty}b_{\ell,n}z^n=0$ for every $\ell\geq m$: suppose otherwise. Then there exists the smallest number $m_1\in\mathbb N^{*}$ such that $b_{m_1}\nequiv0$ and $m_1\geq m$. By the same argument as above, we may assume that $b_{m_1}(z)\equiv i\tilde\beta_1 z+o(|z|)$ for some $\tilde\beta_1\in\mathbb R^{*}$ on $\triangle_{\epsilon_0}$. Moreover, we indeed have $b_{m_1}(z)=i\tilde\beta_1 z$ for some $\tilde\beta_1\in\mathbb R^{*}$: suppose otherwise. Then there exist $k_0\geq2$ and $\tilde c_{k_0}\in\mathbb C^{*}$ such that
\begin{equation}\label{b-m1 in appendix-case2}
b_{m_1}(z)=i\tilde\beta_1 z+\tilde c_{k_0}z^{k_0}+o({|z|}^{k_0}).
\end{equation}Considering $\ell=m_1$ separately in \eqref{tangency condition for s-nonminimal} and then using \eqref{conditions on a in Appendix}, we get 
\begin{equation}\label{refined condition in Case2-appendix}
\begin{split}
&\Re\left[\left(\frac{1}{2i}-\frac{mt^{m-1}P(z)}{2}\right)\sum_{j=0}^{\infty}a_{j,0}{(t+it^mP(z))}^j-t^{m}P_{z}(z)\sum_{n=0}^{\infty}b_{m_1,n}{(t+it^m P(z))}^{m_1}z^n\right.\\
&\quad\left.-t^mP_{z}(z)\sum_{\ell>m_1}b_{\ell}(z){(t+it^mP(z))}^{\ell}\right]\\
&=0
\end{split}
\end{equation}for all $z\in\triangle_{\epsilon_0}$ and $t\in\mathbb R$ sufficiently small. Considering the coefficient of $t^{m+m_1}$ in \eqref{refined condition in Case2-appendix} and then using \eqref{conditions on a in Appendix} and \eqref{b-m1 in appendix-case2}, one can get
\begin{equation}\label{t.c. in appendix for m+m1}
\Re\left[\frac{1}{2}P(z)(m_1+1-m)a_{m_1+1,0}-P_{z}(z)(i\tilde\beta_1 z+\tilde c_{k_0}z^{k_0}+o({|z|}^{k_0}))+o(|P(z)|)\right]\equiv0
\end{equation}on $\triangle_{\epsilon_0}$. Considering again the coefficient of $t^{m+m_1}$ in \eqref{refined condition in Case2-appendix} and then using \eqref{conditions on a in Appendix}, for some $\tilde n_0\in\mathbb N^{0}$ we have
\begin{equation}\label{second refined condition in Case2-appendix}
\Re\left[\frac{1}{2}P(z)(m_1+1-m)a_{m_1+1,0}-P_{z}(z)b_{m_1,\tilde n_0}(z^{\tilde n_0}+o(|z|^{\tilde n_0}))+o(|P(z)|)\right]\equiv0
\end{equation}on $\triangle_{\epsilon_0}$. If $a_{m_1+1,0}\neq0$, then Proposition~\ref{technical lemma1:THM1} yields $\tilde n_0=1$ and $b_{m_1,\tilde n_0}=i\tilde\beta_1$. In addition, if $a_{m_1+1,0}=0$, then \eqref{second refined condition in Case2-appendix} leads to a contradiction to Proposition~\ref{technical lemma1:THM1}. Combining this with the subtraction of \eqref{second refined condition in Case2-appendix} from \eqref{t.c. in appendix for m+m1}, we get
\[
\Re\left[P_{z}(z)\tilde c_{k_0}(z^{k_0}+o({|z|}^{k_0}))\right]\equiv0
\]on $\triangle_{\epsilon_0}$, which contradicts to \cite[Corollary~$4$]{HN2016}. 

Altogether, in this case, we obtain $h_2(w,z)\equiv i\tilde\beta z$ and $P(z)\equiv P(|z|)$ for some $\tilde\beta\in\mathbb R^{*}$ on $\triangle_{\epsilon_0}$.

\medskip

\noindent{\bf Case~3.} $m_0\geq m$. Considering the coefficient of $t^{m+m_0}$ in \eqref{tangency condition for s-nonminimal} and then using \eqref{conditions on a in Appendix}, we get
\begin{equation}\label{Case3-Appendix, tangency condition}
\begin{split}
&\Re\left[\frac{(m_0+1-m)}{2}a_{m_0+1,0}P(z)\right.\\
&\left.\quad-P_{z}(z)\left(b_{m_0,n_0}z^{n_0}+o({|z|}^{n_0})+i(m_0+1-m) P(z)b_{m_0+1-m}(z)+o(|P(z)|)\right)\right]\\
&\equiv O({|P(z)|}^{2})
\end{split}
\end{equation}on $\triangle_{\epsilon_0}$. 

Since $O({|P(z)|}^{2})/P(z)\in o(1)$ and $m_0+1-m\neq0$, if $\Re(a_{m_0+1,0})=a_{m_0+1,0}=0$, then our situation reduces to 
$(\mathrm{E2})$ in \cite[Lemma~$3$]{KN2015}, which leads to a contradiction. Hence, in this case, we must have $a_{m_0+1,0}\neq0$. Moreover, by \cite[Lemma~$3$]{KN2015}, one can get $n_0=1$ and $b_{m_0,1}=i\beta_2$ for some $\beta_2\in\mathbb R^{*}$. Then \eqref{Case3-Appendix, tangency condition} yields
\begin{equation}\label{Case3-Appendix, a relation to make an ode}
\Re\left[i\beta_2 zP_{z}(z)\right]\equiv\left(\tilde\delta+\tilde\epsilon(z)\right)P(z),
\end{equation}where $\tilde\delta:=\mathrm{Re}((m_0+1-m)a_{m_0+1,0}/2)$ and $\tilde\epsilon:\triangle_{\epsilon_0}\rightarrow\mathbb R$ is a smooth function with the condition that $\tilde\epsilon(z)\rightarrow0$ as $z\rightarrow0$. Without loss of generality, we may assume that $\tilde\delta<0$ and $|\tilde\epsilon(z)|<|\tilde\delta|/2$ on $\triangle_{\epsilon_0}$.

Let $r\in(0,\epsilon_0)$ such that $P(r)\neq0$. Then we let $\gamma:[t_0,+\infty)\rightarrow\mathbb C$ such that $\gamma'(t)=i\beta_2\gamma(t)$ and $\gamma(t_0)=r$. Then setting $u(t)=\frac{1}{2}\log|P(\gamma(t))|$, \eqref{Case3-Appendix, a relation to make an ode} shows that $u'(t)=\tilde\delta+\tilde\epsilon(\gamma(t))$. Hence, we get
\begin{equation}\label{Case3-Appendix, FTO}
u(t)-u(t_0)=\tilde\delta(t-t_0)+\int_{t_0}^{t}\tilde\epsilon(\gamma(\tau))d\tau,\;\forall t\geq t_0.
\end{equation}This implies that $u(t)\rightarrow-\infty$ as $t\rightarrow\infty$, and hence $\gamma(t)\rightarrow0$ as $t\rightarrow+\infty$. 

On the other hand, $\gamma'(t)=i\beta_2\gamma(t)$ and $\gamma(t_0)=r$ imply that $\gamma(t)=r\exp(i\beta_2 t)$. Then we get $\gamma(t)\nrightarrow 0$ as $t\rightarrow+\infty$, which contradicts to the above discussion right after \eqref{Case3-Appendix, FTO}.

\medskip

Now we shall show that 
\[
h_1(w,z)\equiv0,
\]if $P(z)\equiv P(|z|)$ and $h_2(w,z)\equiv i\tilde\beta z$ for some $\tilde\beta\in\mathbb R^{*}$ on $\triangle_{\epsilon_0}$. Suppose otherwise. Then it follows from \eqref{conditions on a in Appendix} that there exists the smallest number $j_0\in\mathbb N^{*}$ such that $a_{j_{0},0}\neq0$. Since $P(z)\equiv P(|z|)$ and $h_2(w,z)\equiv i\tilde\beta z$ for some $\tilde\beta\in\mathbb R^{*}$ on $\triangle_{\epsilon_0}$, using \eqref{tangency condition for s-nonminimal} and \eqref{conditions on a in Appendix}, one can deduce that for all $t\in\mathbb R$ sufficiently small
\begin{equation}\label{Case3-Appendix,special tc}
\Re\left[\left(\frac{1}{2i}-\frac{mt^{m-1}P(z)}{2}\right)\sum_{j=0}^{\infty}a_{j,0}{(t+i t^{m}P(z))}^{j}\right]\equiv0
\end{equation}on $\triangle_{\epsilon_0}$. For a fixed $t\in\mathbb R$ small enough, we may regard \eqref{Case3-Appendix,special tc} as a polynomial of a variable $P(z)$ on $\triangle_{\epsilon_0}$. Collecting the terms of degree $1$ with respect to $P(z)$ in \eqref{Case3-Appendix,special tc}, for a fixed $t\in\mathbb R$ small enough, we get
\[
P(z)\Re\left[\sum_{j=0}^{\infty}\frac{1}{2}(j-m)a_{j,0}t^{m+j-1}\right]\equiv0
\]on $\triangle_{\epsilon_0}$. Since the connected component of $z=0$ in the zero set of $P$ is $\{0\}$, we have 
\[
\Re\left[\sum_{j=0}^{\infty}\frac{1}{2}(j-m)a_{j,0}t^{m+j-1}\right]=0
\]for a fixed $t\in\mathbb R$ sufficiently small. Moreover, since $t$ can be chosen arbitrarily small, we first have
\[
\Re\left(a_{j,0}\right)=0,\;\forall j\in\mathbb N^{*}\setminus\{m\}.
\]Combining this with the second relation of \eqref{conditions on a in Appendix}, one can deduce that
\begin{equation}\label{pre-conclusion on a in Appendix}
a_{j,0}=0,\;\forall j\in\mathbb N^{*}\setminus\{m\}.
\end{equation}Then it follows from \eqref{conditions on a in Appendix} and \eqref{pre-conclusion on a in Appendix} that $a_{m,0}$ can be a unique candidate to be non-zero among all the possible $a_{j,k}$'s. For this reason, we now assume that
\[
H(w,z):=a_{m,0}w^{m}\partial_{w}+i\tilde\beta z\partial_{z}\in\mathfrak{aut}_{0}(M_{P,m},0)
\]for some $a_{m,0}\in\mathbb R$ and $\tilde\beta\in\mathbb R^{*}$. 

Let us denote by $\{\varphi_{t}\}_{t\in\mathbb R}:=\{(\varphi^{1}(t),\varphi^{2}(t))\}_{t\in\mathbb R}\subset\mathrm{Aut}(M_{P,m},0)$ the $1$-parameter subgroup generated by the vector field $H$, that is, for $t\in\mathbb R$
\[
\frac{d\varphi^{1}}{dt}(t)=a_{m,0}{(\varphi^{1}(t))}^{m};\;\frac{d\varphi^{2}}{dt}(t)=i\tilde\beta\varphi^{2}(t)
\]with $(\varphi^{1}(0),\varphi^{2}(0))=(w,z)\in M_{P,m}$. However, we note that the solution $\varphi_{t}$ is not invertible, if $m>1$ and $a_{m,0}\neq0$. This tells us that if $m>1$ and $H\in\mathfrak{aut}_{0}(M_{P,m},0)$, then we should have $a_{m,0}=0$.

Altogether, we conclude that if $P(z)\equiv P(|z|)$ and $h_{2}(w,z)\equiv i\tilde\beta z$ for some $\tilde\beta\in\mathbb R^{*}$ on $\triangle_{\epsilon_0}$, then $h_{1}(w,z)\equiv0$ holds, as desired.

\bigskip

\begin{acknowledgement}Part of this work was done while the first and last authors were visiting the Vietnam Institute for Advanced Study in Mathematics (VIASM). They would like to thank the VIASM for financial support and hospitality. The first author was supported by the Vietnam National Foundation for Science and Technology Development (NAFOSTED) under grant number 101.02-2017.311. The last author was supported by the National Research Foundation of Korea under grant number NRF-2018R1D1A1B07044363.
\end{acknowledgement}

\bibliographystyle{plain}

\end{document}